\documentclass[a4paper,10pt,reqno]{smfart}
\usepackage[francais]{babel}
\usepackage[T1]{fontenc}
\usepackage[utf8]{inputenc}
\usepackage{lmodern}
\usepackage{amsmath} 
\usepackage{amsfonts}
\usepackage{amscd} 
\usepackage{amssymb} 
\usepackage{enumerate}
\usepackage{mathrsfs}
\usepackage[all]{xy}
\makeatletter
\@addtoreset{equation}{subsection}
\renewcommand{\theequation}{\arabic{section}.\arabic{subsection}.\arabic{equation}}
\makeatother
\newcommand{\termin}[1]{{\em #1}}   
\newcommand{\bA}{\mathbf{A}}\newcommand{\bC}{\mathbf{C}}\newcommand{\bD}{\mathbf{D}}
\newcommand{\bE}{\mathbf{E}}\newcommand{\bF}{\mathbf{F}}

\newcommand{\bN}{\mathbf{N}}
\newcommand{\bQ}{\mathbf{Q}}\newcommand{\bR}{\mathbf{R}}

\newcommand{\bZ}{\mathbf{Z}}

\newcommand{\ba}{\boldsymbol{a}}\newcommand{\bd}{\boldsymbol{d}}
\newcommand{\be}{\boldsymbol{e}}\newcommand{\bbf}{\boldsymbol{f}}\newcommand{\bg}{\boldsymbol{g}}

\newcommand{\bm}{\boldsymbol{m}}
\newcommand{\bt}{\boldsymbol{t}}

\newcommand{\balpha}{\boldsymbol{\alpha}}

\newcommand{\bgamma}{\boldsymbol{\gamma}}

\newcommand{\bnu}{\boldsymbol{\nu}}
\newcommand{\bmu}{\boldsymbol{\mu}}
\newcommand{\cC}{{\mathcal C}}\newcommand{\cD}{{\mathcal D}}
\newcommand{\cF}{{\mathcal F}}\newcommand{\cG}{{\mathcal G}}\newcommand{\cH}{{\mathcal H}}

\newcommand{\cP}{{\mathcal P}}

\DeclareMathAlphabet{\eulercal}{U}{eus}{m}{n}

\newcommand{\ecC}{{\eulercal C}}\newcommand{\ecD}{{\eulercal D}}
\newcommand{\ecE}{{\eulercal E}}\newcommand{\ecF}{{\eulercal F}}
\newcommand{\ecG}{{\eulercal G}}

\newcommand{\ecO}{{\eulercal O}}

\DeclareMathAlphabet{\beulercal}{U}{eus}{b}{n}

\newcommand{\becD}{{\beulercal D}}
\newcommand{\becE}{{\beulercal E}}\newcommand{\becF}{{\beulercal F}}
\newcommand{\becG}{{\beulercal G}}

\newcommand{\mfI}{\mathfrak{I}}

\newcommand{\scC}{\mathscr{C}}\newcommand{\scD}{\mathscr{D}}
\newcommand{\scE}{\mathscr{E}}\newcommand{\scF}{\mathscr{F}}\newcommand{\scG}{\mathscr{G}}
\newcommand{\scI}{\mathscr{I}}\newcommand{\scK}{\mathscr{K}}
\newcommand{\scM}{\mathscr{M}}\newcommand{\scN}{\mathscr{N}}\newcommand{\scO}{\mathscr{O}}\newcommand{\scP}{\mathscr{P}}
\newcommand{\scT}{\mathscr{T}}

\newcommand{\imply}{\Rightarrow}
\newcommand{\longto}{\longrightarrow}

\DeclareMathOperator{\Pic}{Pic}

\DeclareMathOperator{\Div}{Div}

\DeclareMathOperator{\dist}{dist}

\DeclareMathOperator{\rg}{rg}

\DeclareMathOperator{\rk}{rk}

\DeclareMathOperator{\Min}{Min}
\DeclareMathOperator{\Max}{Max}

\DeclareMathOperator{\Spec}{Spec}

\DeclareMathOperator{\pgcd}{pgcd}

\DeclareMathOperator{\Cox}{Cox}

\DeclareMathOperator{\Vol}{Vol}

\let\leq\leqslant
\let\geq\geqslant
\newcommand{\sumu}[1]{\underset{#1}{\sum}}
\newcommand{\sumsub}[1]{\sum_{\substack{#1}}}

\newcommand{\produ}[1]{\underset{#1}{\prod}}
\newcommand{\prodsub}[1]{\prod_{\substack{#1}}}
\newcommand{\cupu}[1]{\underset{#1}{\cup}}
\newcommand{\capu}[1]{\underset{#1}{\cap}}

\newcommand{\disju}[1]{\underset{#1}{\bigsqcup}}

\newcommand{\Maxu}[1]{\underset{#1}{\Max}}
\newcommand{\Minu}[1]{\underset{#1}{\Min}}

\newcommand{\pgcdu}[1]{\underset{#1}{\pgcd}}
\newcommand{\eps}{\varepsilon}
\newcommand{\vide}{\varnothing}

\newcommand{\eqdef}{\overset{\text{{\tiny{déf}}}}{=}}

\newcommand{\ind}{\mathbf{1}}

\newcommand{\wt}{\widetilde}

\newcommand{\acc}[2]{\left\langle #1\, ,\,#2 \right\rangle} 
\newcommand{\norm}[1]{\left|\left| #1 \right|\right|} 
 
\newcommand{\abs}[1]{\left| #1 \right|} 
 
\newcommand{\adh}[1]{\overline{#1}}

\newcommand{\bigou}[2]{\underset{{#1}}{\ecO}\left(#2\right)} 
 
\newcommand{\ceff}{C_{\text{\textnormal{\tiny{eff}}}}}

\newcommand{\cf}{{\it cf.}\  }

\newcommand{\eg}{{\it e.g.}\ }
\newcommand{\opcit}{{\it op.cit.}}

\newtheorem{thm}{Th\'eor\`eme}[section]
\newtheorem{prop}[thm]{Proposition}
\newtheorem{lemme}[thm]{Lemme}
\newtheorem{cor}[thm]{Corollaire}

\newtheorem{hyp}[thm]{Hypoth\`ese}

\theoremstyle{definition}
\newtheorem{defi}[thm]{D\'efinition}

\theoremstyle{remark}

\newtheorem{ex}[thm]{Exemple}

\newtheorem{rem}[thm]{Remarque}

\DeclareMathOperator{\HOM}{\textbf{Hom}}
\usepackage{url}
\newcommand{\sumsuby}[1]{\sum^y_{\substack{#1}}}
\newcommand{\scan}[1]{s_{#1}} 
\newcommand{\idex}{\mathscr{I}_X}

\newcommand{\jz}{j_{_{0}}}
\newcommand{\ju}{j_{_{1}}}
\newcommand{\jd}{j_{_{2}}}
\newcommand{\iz}{i_{_{0}}}
\newcommand{\iu}{i_{_{1}}}
\newcommand{\ideux}{i_{_{2}}}

\newcommand{\cpr}{\scC_{p,r}}
\newcommand{\cprz}{\scC_{p,r}^{(0)}}
\newcommand{\gc}{g_{{}_{\courbe}}}
\newcommand{\diveff}{\Div_{\text{\textnormal{eff}}}}

\newcommand{\diveffcpr}{\diveff(\cpr)}
\newcommand{\courbe}{\scC}
\newcommand{\courbez}{\scC^{(0)}}
\newcommand{\diveffc}{\diveff(\courbe)}

\newcommand{\cprinc}{c_{\text{\textnormal{princ}}}} 
\newcommand{\hc}{h_{{}_{\courbe}}}
\newcommand{\tors}{\scT}
\newcommand{\fonc}{\scN}
\newcommand{\foncpr}{\scN_{_{p,r}}}

\newcommand{\foncjzpr}{\scN_{_{\jz,p,r}}}

\newcommand{\indsec}{\mfI}
\newcommand{\OdeC}{\struct{\courbe}}
\newcommand{\struct}[1]{\scO_{#1}}
\newcommand{\can}[1]{\scK_{#1}}
\newcommand{\Dtot}{\scD_{\text{tot}}}
\newcommand{\classadm}{\ecC}
\newcommand{\classadminc}{\classadm_{\text{inc}}}
\newcommand{\muX}{\mu_{_X}}
\newcommand{\muXpr}{\mu_{_{X,p,r}}}
\newcommand{\muzX}{\mu^{\circ}_{_X}}
\newcommand{\nuXJ}{\nu_{_{X,J}}}
\newcommand{\nuXJpr}{\nu_{_{X,J,p,r}}}
\newcommand{\nuzXJ}{\nu^{\circ}_{_{X,J}}}

\newcommand{\Nindprim}{\bN^{\indsec}_{\text{inc}}}

\newcommand{\fpr}{\bF_{\!p^{\scriptscriptstyle r}}}
\title{Exemples de comptages de courbes sur les surfaces}
\author{David Bourqui}
\address{IRMAR\\Université de Rennes 1\\Campus de Beaulieu\\35042
  Rennes Cedex\\France}
\email{david.bourqui@univ-rennes1.fr}
\begin{document}
\maketitle
\begin{abstract}
Soit $X$ une surface dont l'anneau de Cox a une seule relation, laquelle
vérifie en outre une certaine propriété de linéarité.
Nous montrons, sous une hypothèse simple, que les conjectures de Manin
géométriques valent pour certains degrés du cône effectif dual de $X$
(notamment pour ces degrés l'espace de modules de morphismes a la
dimension attendue).
Le résultat s'applique  à une classe de
surfaces de del Pezzo généralisées qui a été intensément étudiée dans
le cadre de la conjecture de Manin arithmétique.

~

\textbf{{\em Some examples of curves countings on surfaces.}} 
Let $X$ be a surface whose Cox ring has a single relation 
satisfying moreover a kind of linearity property. Under a simple
assumption, we show that the geometric Manin's conjectures hold for
some degrees lying in the dual of the effective cone of $X$
(in particular, for those degrees the moduli space of morphisms has the expected dimension).
The result applies to a class of generalized del Pezzo
surfaces which has been intensively studied in the context of the
arithmetic Manin's conjecture.
\end{abstract}
\section{Introduction}\label{sec:intro}
Dans ce texte, nous poursuivons notre investigation de la
conjecture de Manin sur le comportement asymptotique du nombre de
courbes de degr\'e born\'e sur les
\og hypersurfaces intrinsèques linéaires\fg, c'est-à-dire 
les variétés dont l'anneau total de coordonn\'ees s'identifie \`a l'anneau d'une
hypersurface affine dont l'équation vérifie en outre une certaine propriété de linéarité.
Nous renvoyons à la section \ref{sec:Cox} pour plus de précisions sur la
terminologie. Signalons d'ores et déjà que la classification de
Derenthal des surfaces de del Pezzo généralisées de degré $\geq 3$
dont l'anneau total de
coordonnées n'a qu'une équation (\cf \cite{Der:sdp:ut:hyp}) montre que les 21 surfaces
qu'il obtient tombent toutes dans la classe considérée, sauf
une (l'une des deux surfaces cubiques avec singularité $\bD_4$). Ces surfaces ont été
intensément étudiées dans le cadre de la conjecture de Manin
arithmétique (évaluation asymptotique du nombre de points rationnels
de hauteur bornée lorsque le corps de base est le corps des
rationnels), \cf \cite{dlBBD,dlBBr:sdP4I,DerTsc:univ2007,BroDer:DP4A4,Der:int,BroDer:dP3D5,Lou:dP6A2:2010,LeBou:dP43A1A1A2:2012,LeBou:dP32A2A1:2011}

Le résultat principal obtenu dans cet article est le suivant.
Nous l'énonçons dans cette introduction sous une forme volontairement simplifiée.
\begin{thm}
Soit $k$ un corps parfait et $\courbe$ une $k$-courbe projective, lisse et g\'eom\'etriquement
int\`egre de genre $\gc$. 

Soit $X$ une surface projective, lisse 
et g\'eom\'etriquement int\`egre d\'efinie sur $k$ qui est une
hypersurface intrinsèque linéaire, et qui vérifie en outre l'hypothèse
\ref{hyp:noninter} ci-dessous.

Alors $X$ satisfait une version partielle des conjectures de Manin géométriques.
\end{thm}
Le résultat est donné sous sa forme précise par le théorème \ref{thm:main}.
Nous nous contentons ici de quelques commentaires. 

Tout d'abord, dans cet article, l'appellation \og conjectures de Manin géométriques \fg\
recouvrira d'une part un cadre 
où le corps $k$ est supposé fini
et où on s'intéresse à la problématique du comptage asympotique de courbes de
degré borné, et d'autre part un cadre où le corps $k$ peut être a
priori quelconque, et où on étudie la dimension et le nombre de composantes des espaces de
modules de morphismes de courbes de degré donné. 
La section \ref{sec:conj} contient plus de détails sur la forme
précise des conjectures étudiées dans ce texte. 

Ensuite, la partialité évoquée dans l'énoncé ci-dessus
est explicitement quantifiable~: le théorème dit en gros que
pour les cas considérés, la conjecture de Manin vaut sur une certaine partie explicite du dual
$\ceff(X)^{\vee}$ du cône effectif de $X$. Par exemple le résultat dit que sur cette partie, l'espace de modules
de morphismes est irréductible et a la dimension attendue.

Pour deux surfaces de la liste de Derenthal (en fait pour trois, mais
pour la troisième le résultat était déjà connu d'après nos travaux
précédents) on obtient les conjectures de Manin \og complètes\fg,
c'est-à-dire que les conjectures
valent sur $\ceff(X)^{\vee}$,
et pour cinq autres on obtient une version \og complète à la limite\fg\ de
ces conjectures (on peut trouver des parties --en fait des réunions de
cônes-- sur lesquelles les conjectures valent et dont le volume est aussi proche que l'on veut de
celui de $\ceff(X)^{\vee}$), \cf le corollaire \ref{cor:manin:delpezzo}. Pour toutes les autres, on peut calculer
explicitement la proportion de $\ceff(X)^{\vee}$ sur laquelle notre
méthode montre que les conjectures valent, \cf la table \ref{tab:results}.
En particulier, les résultats de \cite{Bou:moduli} (qui concernaient
le cadre de l'étude de la dimension des espaces de modules)
sont améliorés. Nous renvoyons à la section \ref{sec:res:main} pour plus de précisions.

Enfin, il faut souligner que le théorème \ref{thm:main} isole le terme principal attendu et
identifie explicitement le \og terme d'erreur\fg\ qui est la cause de
la partialité dans l'énoncé ci-dessus. Il montre en que les
conjectures complètes sont vérifiées si et seulement si ce dernier terme
peut être contrôlé de manière ad hoc. La méthode employée
dans ce texte ne nous permet d'obtenir ce contrôle que sur certaines régions du dual
du cône effectif, celles où les coefficients de l'équation de l'anneau
de coordonnées total de $X$ sont de degré \og suffisamment
grand\fg. Si cette dernière condition n'est pas vérifiée, il
faut tenir compte des syzygies en bas degré, ce qui semble délicat
(mais peut-être pas insurmontable).

Le schéma général de la  stratégie  employée dans cet article est le
même que pour \cite{Bou:compt,Bou:fam,Bou:moduli}~: relèvement du problème de comptage en
termes de l'anneau total de coordonnées de $X$, inversion de Möbius,
comptage de sections globales vérifiant certaines équations linéaires
à l'aide de Riemann-Roch et d'algèbre linéaire élémentaire et étude
des séries génératrices en découlant.

Cependant, dans le présent travail, l'utilisation d'une
inversion de Möbius \og partielle\fg\ et non plus totale jointe à une étude plus fine des séries
génératrices intervenant permettent d'affaiblir et de simplifier de manière significative les hypothèses qui
étaient nécessaires dans \cite{Bou:compt,Bou:fam,Bou:moduli}, dont certaines exigeaient des
calculs lourds et peu éclairants pour leur éventuelle vérification, et
n'étaient d'ailleurs pas vérifiées pour toutes les surfaces de la
liste de Derenthal. 
La seule hypothèse nécessaire est ici l'hypothèse \ref{hyp:noninter}, 
qui est très simple à vérifier pour peu qu'on ait une description
explicite de l'anneau de coordonées total de la variété considérée.
On voit ainsi aussitôt qu'elle est satisfaite par les 20 surfaces
de la liste de Derenthal. Peut-être pourrait-t-on d'ailleurs se risquer à espérer qu'elle est satisfaite
pour toutes les hypersurfaces intrinsèques linéaires. On montre
notamment sous cette seule hypothèse que le terme principal de la
fonction zêta des hauteurs correspond bien à celui prédit par Peyre
(\cf la relation \eqref{eq:sum:becG:gamma}).
Dans \cite{Bou:fam}, cette propriété nécessitait a priori une
vérification calculatoire au cas par cas (\cf la remarque 11 de \opcit).

Il est à noter que le résultat final est démontré pour les surfaces,
mais que beaucoup d'arguments, notamment ceux utilisés pour le calcul
du terme principal, sont valables en toute dimension, et sont donnés
dans ce cadre. Il est sans doute possible, au prix de complications
techniques, d'étendre les résultats obtenus en dimension supérieure.
Il faut remarquer que peu d'exemples explicites d'hypersurfaces
intrinsèques linéaires en dimension supérieure semblent connus.

\section{Conjectures de Manin g\'eom\'etriques}\label{sec:conj}
Dans tout cet article, $k$ désigne un corps parfait et $\courbe$ une $k$-courbe projective, lisse et g\'eom\'etriquement
int\`egre, dont on note $\gc$ le genre. Soit $\diveffc$ le monoïde
des diviseurs effectifs de $\courbe$ définis sur $k$.
On note $\diveffc_{\leq 1}$ le sous-ensemble 
de $\diveffc$ constitué des diviseurs dont toutes
les multiplicités valent au plus $1$.

Soit $X$ une $k$-vari\'et\'e projective, lisse 
et g\'eom\'etriquement int\`egre d\'efinie sur $k$. 
On suppose que son groupe de Picard g\'eom\'etrique est libre de rang fini
et d\'eploy\'e, c'est-à-dire que l'action du groupe de Galois absolu est triviale.
On suppose en outre que $-\can{X}$, la classe du faisceau anticanonique de $X$ dans le groupe de Picard, 
est situ\'ee \`a l'int\'erieur du c\^one effectif $\ceff(X)$ de $X$.
On suppose par ailleurs que $\ceff(X)$ est finiment engendré.

Pour $U$ ouvert de Zariski non vide de $X$ assez petit 
et pour tout  élément $y$ du dual  $\Pic(X)^{\vee}$ du groupe de
Picard de $X$ soit $\HOM_U(\courbe,X,y)$
la $k$-variété quasi-projective paramétrisant
les morphismes $f\,:\,\courbe\to X$ 
dont l'image rencontre $U$
et de degré absolu $[f_{\ast}\courbe]=y$.

Si $k$ est un corps fini de cardinal $q$,
on note $\hc$ le nombre de classes de diviseurs de degr\'e $0$
et $\courbe^{(0)}$ l'ensemble des points ferm\'es de $\courbe$. 
Pour $v\in \courbe^{(0)}$, on note
$\kappa_v$ le corps r\'esiduel et $q_v=q^{\,f_v}$ son cardinal.
On pose alors
\begin{equation}\label{eq:def:gamma}
\gamma(X)
\eqdef
\left(\frac{\hc\,q^{(1-\gc)}}{q-1}\right)^{\rg(\Pic(X))}\:
q^{(1-\gc)\,\dim(X)}\,
\prod_{v\in\courbe^{(0)}}
(1-q_v^{-1})^{\rg(\Pic(X))}\,\frac{\#{X(\kappa_{v})}}{q_v^{\,\dim(X)}}.
\end{equation}
L'interprétation conceptuelle de $\gamma(X)$ se fait en termes du
volume de l'espace adélique associé à la variété $X\times_k k(\courbe)$ pour une certaine mesure de
Tamagawa, \cf \cite{Pey:duke,Pey:var_drap}. Sous les seules hypothèses que nous avons données, la
convergence du produit eulérien dans \eqref{eq:def:gamma} n'est a
priori pas assurée. Nous renvoyons à \cite{Pey:var_drap} pour des précisions
sur les hypothèses permettant de montrer la convergence en
utilisant Weil-Deligne. Pour la classe de
variétés étudiée dans cet article, ces hypothèses sont vérifiées et la
convergence du produit peut d'ailleurs se voir de manière élémentaire.

\begin{defi}
Soit $(a_n)\in \bC^{\bN}$, $\rho>0$ et $k\in \bN$.
On dit que la série $\sum a_n\,t^n$ est $\rho$-contrôlée à l'ordre $r$
s'il existe $(b_n)\in \bR^{\bN}$ telle que $\abs{a_n}\leq b_n$,
le rayon de convergence de $\sum b_nt^n$ est sup\'erieur \`a $\rho$
et sa somme se prolonge en une fonction m\'eromorphe
sur un disque de rayon strictement sup\'erieur  \`a $\rho$,
ayant des p\^oles d'ordre au plus $r$ sur le cercle de rayon $\rho$.
\end{defi}
Voici les versions de la conjecture de Manin géométrique que nous
considèrerons dans cet article. 
Pour nous, un cône sera toujours
polyedral rationnel.
\begin{defi}
Soit $U$ un ouvert non vide de $X$ et $\cP$ une partie de
$\ceff(X)^{\vee}$ qui est une réunion finie de cônes.
\begin{enumerate}
\item
On suppose que $k$ est un corps fini. 
On dit que $(X,U)$ vérifie 
\og Manin I sur $\cP$\fg\
si la série
\begin{equation}
\sum_{y\in \cP\cap \Pic(X)^{\vee}}\,
(\# \HOM_U(\courbe,X,y)(k)
-\gamma(X)
q^{\,\acc{y}{-\can{X}}}
)
t^{\,\acc{y}{-\can{X}}}
\end{equation}
est $q^{-1}$-contr\^ol\'ee \`a l'ordre $\rg(\Pic(X))-1$,
et on dit que $(X,U)$ vérifie 
\og Manin II sur $\cP$\fg\
si on a
\begin{equation}
\lim_{
\substack{
y\in \cP\cap \Pic(X)^{\vee}
\\
\dist(y,\partial \cP)\to +\infty
}
}
q^{\,\acc{y}{\can{X}}}\,\# \HOM_U(\courbe,X,y)(k)=\gamma(X).
\end{equation}
On dit que $(X,U)$ vérifie 
\og Manin I fort sur $\cP$\fg\  si pour toute partie $\cP'$ de $\cP$
qui est une union finie de cônes, $(X,U)$ vérifie \og Manin I sur $\cP'$\fg.
\item
On dit que $(X,U)$ vérifie 
\og Manin géométrique sur $\cP$\fg\ 
s'il existe un réel positif $M$ (dépendant de $X$, $\courbe$ et $\cP$)
telle que pour tout  $y\in \ceff(X)^{\vee}\cap \Pic(X)^{\vee}$ 
vérifiant $\dist(y,\partial \cP)\geq M$, 
la variété $\HOM_U(\courbe,X,y)$ est géométriquement irréductible de
dimension $\acc{y}{-\can{X}}+(1-\gc)\dim(X)$.
\end{enumerate}
Pour $\alpha\in \{I,I\text{ fort},II,\text{géométrique}\}$, on dit que $(X,U)$ vérifie  
\og Manin $\alpha$\fg\ si $(X,U)$ vérifie 
\og Manin $\alpha$ sur $\ceff(X)^{\vee}$\fg.
\end{defi}
\begin{rem}
Notons $\rho(X)=\rg(\Pic(X))$ et, pour toute partie mesurable $\cP$ de
$\ceff(X)^{\vee}$, notons $\alpha_{\cP}(X)$ le volume de l'intersection de
$\cP$ avec l'hyperplan affine $\{\acc{\bullet}{-\can{X}}=1\}$ (la mesure étant
normalisée par le réseau $\Pic(X)^{\vee}$).
Soit $\delta\eqdef \Max\{d\in \bN_{>0},\quad \frac{1}{d}\can{X}\in
\Pic(X)\}$. Grâce aux estimées de Cauchy, on voit que \og Manin I sur $\cP$\fg\
équivaut à l'asymptotique suivante pour le nombre de
courbes de degré anticanonique borné~:
\begin{equation}
\sumsub{
y\in \cP\cap \Pic(X)^{\vee}
\\
\acc{y}{-\can{X}}=\delta\,d
}
\!\!\!\!\!
\!\!\!\!\!
\# \HOM_U(\courbe,X,y)(k)
=
\alpha_{\cP}(X)\gamma(X)\,
d^{\rho(X)-1}
q^{\,\delta\,d}
+
\bigou{d\to +\infty}{d^{\rho(X)-2}q^{\delta\,d}}
\end{equation}
Ainsi \og Manin I\fg\ équivaut essentiellement à la conjecture de
Manin originelle (notons que si $\ceff(X)$ est engendré par des
diviseurs de $X\setminus U$, ce qu'on peut toujours supposer si $U$
est assez petit, on a $\HOM_U(\courbe,X,y)=\vide$ si $y\notin
\ceff(X)^{\vee}$). 
\og Manin I fort\fg, tout comme \og Manin II\fg,
expriment une sorte de propriété d'équirépartition des degrés absolus des courbes de
degré anticanonique $d$ dans le cône effectif dual.

Notons qu'a priori aucune des deux propriétés \og Manin I fort\fg\ ou \og Manin II\fg\
n'entraîne l'autre. Cependant dans cet article, les cas que nous en
établirons découleront d'une même propriété plus forte satisfaite par
les $\# \HOM_U(\courbe,X,y)(k)$ (\cf la remarque \ref{rem:pte:plus:forte}).
\end{rem}
\begin{rem}
Pour Manin \og géométrique\fg, il est a priori concevable d'étudier des invariants \og motiviques\fg\
plus fins que la dimension et le nombre de composantes, comme c'est
fait dans \cite{Bou:prod:eul:mot} pour les variétés toriques.
\end{rem}

Rappelons que par une application standard de Weil-Deligne et la théorie
élémentaire de la déformation des morphismes (\cf \eg \cite{Deb:higher}), on a la
propriété suivante (on suppose pour simplifier que $k=\bQ$, mais par
les arguments habituels ce procédé de réduction aux corps finis peut
se réaliser sur un corps de base quelconque).
\begin{lemme}\label{lm:weil}
$\HOM_U(\courbe,X,y)$ est géométriquement irréductible de dimension $\acc{y}{-\can{X}}+(1-\gc)\dim(X)$
si et seulement si pour presque tout nombre premier $p$ on a
\begin{equation}
\lim_{
\substack{
r\to +\infty
}
}
p^{\,r[(\gc-1)\dim(X)+\acc{y}{\can{X}}]}\,\# \HOM_U(\courbe,X,y)(\fpr)
=1
\end{equation}
\end{lemme}

\section{Anneau de Cox, inversion de M\"obius, et rel\`evement au  torseur universel}\label{sec:Cox}

On conserve les notations et hypoth\`eses de la section pr\'ec\'edente.
On suppose en outre que l'anneau total de coordonnées ou anneau de Cox de $X$ 
(\cf \eg \cite{Has:eq:ut:cox:rings}),
not\'e $\Cox(X)$, est de type fini.
Soit $\{u_i\}_{i\in \indsec }$ 
une famille finie de sections globales (non constantes) 
qui engendre $\Cox(X)$. 
Pour $i\in \indsec $, on note
$\scE_i$ le diviseur des z\'eros de $u_i$.
Soit $X_0$ l'ouvert de $X$ \'egal au compl\'ementaire de la r\'eunion des
diviseurs $\{\scE_i\}_{i\in \indsec }$.

\subsection{Inversion de Möbius}
Soit 
\begin{equation}
\classadminc
\eqdef
\{K\subset \indsec,\quad \capu{i\in K}\scE_i\neq \vide\}
\end{equation}
et
\begin{equation}
\Nindprim
=
\{\be\in \bN^\indsec,\quad \{i\in \indsec,\quad e_i\geq 1\}\in \classadminc\}.
\end{equation}
On a la g\'en\'eralisation classique suivante de la formule
d'inversion de M\"obius (\cf \eg \cite[proposition 1.21]{Bou:compt}).
\begin{prop}\label{prop:mu}
Il existe une unique fonction $\muzX\,:\,\bN^{\indsec }\longto \bC$
v\'erifiant
\begin{equation}
\forall \be\in \bN^{\indsec},\quad
\sum_{0\leq \be'\leq \be}\muzX (\be')
=
\ind_{\Nindprim}.
\end{equation}
En particulier, on a $\muzX (\balpha)=0$ dans les cas suivants~:
\begin{enumerate}
\item
il existe $i\in \indsec $ tel que $\alpha_{i}\geq 2$~;
\item
$\balpha$ est non nul et l'intersection 
$\cap_{i,\,\,\alpha_i\neq 0}\,\scE_i$ est non vide~; 
ceci vaut en particulier si
on a 
$\sum_{i\in \indsec }\alpha_i=1$.
\end{enumerate}
\end{prop}
Soit $J$ une partie de $\indsec$ et $I=\indsec \setminus J$.
Pour $\bg\in \bN^J$, $\bbf\in \bN^I$
on pose
\begin{equation}\label{eq:def:nu0}
\nuzXJ(\bg,\bbf)
\eqdef
\sum_{0\leq \bbf'\leq \bbf}
\muzX(\bg,\bbf').
\end{equation}
\begin{rem}\label{rem:nu0nonnul}
Si $(0,\bbf)\notin \Nindprim$, alors 
$\nuzXJ (\bg,\bbf)$ est nul pour tout $\bg$.
\end{rem}
\begin{rem}\label{rem:sum:bbfK:n0}
Pour $K\subset I$, soit $\bbf_{K}\in \{0,1\}^I$ l'élément
caractérisé par $f_{K,i}=1$ si et seulement si $i\in K$.

Soit $K\subset I$ tel qu'on ait, pour tout $j\in J$, 
$\{j\}\cup K\notin \classadminc$.
Alors pour toute partie {\em non vide} $J_1$ de $J$ et tout $\bg_2\in \{0,1\}^{J\setminus J_1}$
on a la relation
\begin{equation}
\sum_{\bg\in \{0,1\}^{J_1}}\nuzXJ (\bg_1,\bg_2,\bbf_{K})=0.
\end{equation}
Ceci découle des définitions de $\nuzXJ $ et $\muzX$ ainsi que d'une
récurrence sur $\abs{\bg_2}$.

Par ailleurs il découle facilement de la proposition \ref{prop:mu} que pour tout $\bbf\in \bN^I$ tel
que $\{i\in I,\quad f_i\geq 1\}=K$, on a $\nuzXJ(\bg,\bbf)=\nuzXJ(\bg,\bbf_K)$.
\end{rem}

Pour $\becD\in \diveffc^{\indsec}$, on pose 
\begin{equation}
\muX(\becD)\eqdef\prod_{v\in \courbez}\muzX(v(\becD)).
\end{equation}
Pour $\becG\in \diveffc^J$ et $\becD\in \diveffc^I$, on pose 
\begin{equation}\label{eq:def:nuXJ}
\nuXJ (\becG,\becD)
\eqdef
\sum_{0\leq \becD'\leq \becD}
\muX(\becG,\becD')
=\prod_v \nuzXJ (v(\becG,\becD)).
\end{equation}
\begin{rem}\label{rem:nuXJnul}
D'après la proposition \ref{prop:mu}, $\nuXJ (\becG,\becD)$ est nul
dès que $\becG\notin \diveffc^J_{\leq 1}$.
\end{rem}
Si $k=\bQ$, on désignera par $\muXpr$ et $\nuXJpr$ les fonctions sur $\diveffcpr^{\indsec}$
obtenues en remplaçant dans les définitions ci-dessus $\courbe$ par $\cpr$.

\subsection{Hypersurfaces intrinsèques linéaires}\label{subsec:hypintlin}
On conserve les hypoth\`eses et notations 
introduites pr\'ec\'edemment.
On dit que $X$ est une \termin{hypersurface intrinsèque linéaire}
si $\idex$ est principal et 
la propriété suivante est vérifiée~: il existe $J\subset \indsec$
et une famille $\{I_j\}_{j\in J}$ de parties non vide deux à deux
disjointes de $\indsec\setminus J$ 
tels que les classes $\{\ecE_i\}_{i\in \indsec\setminus J}$ 
forment une base de $\Pic(X)$
et, après une renormalisation éventuelle des $\{s_i\}_{i\in \indsec}$, 
un générateur de $\idex$ est de la forme
\begin{equation}\label{eq:cF}
\cF=\sum_{j\in J} s_j\,\prod_{i\in I_j}s_i^{b_{i,j}},\quad b_{i,j}\in \bN_{>0}.
\end{equation}
Tout $J\subset \indsec$ 
vérifiant la propriété précédente est appelé un \termin{choix admissible de variables}.
\begin{ex}
D'après \cite{Der:sdp:ut:hyp}, les 21 surfaces de del Pezzo
généralisées de degré $\geq 3$ dont l'anneau total de coordonnées n'a qu'une équation
sont toutes, à une exception près (l'une des deux cubiques
avec singularité $\bD_4$ apparaissant dans la classification),  des
hypersurfaces intrinsèques linéaires. Dans \cite{Bou:fam} est
construite une famille d'hypersurfaces intrinsèques linéaires de
dimension non bornée.
\end{ex}
On notera $\Dtot\in \Pic(X)$ le degré d'un générateur de $\idex$.
D'apr\`es \cite[proposition 8.5]{BerHau:Cox}, on a la formule d'adjonction suivante.
\begin{lemme}
La classe du fibr\'e anticanonique de $X$ est
\begin{equation}\label{eq:adj}
-\can{X}=\sum_{i\in \indsec}\scE_i-\Dtot.
\end{equation}
\end{lemme}

\subsection{Nombre de points sur le torseur universel}
Soit $X$ une hypersurface intrinsèque linéaire.
On note $\tors_{X}$ le schéma affine $\Spec(\bZ[s_i]_{i\in \indsec}/\cF)$
où $\cF$ est de la forme \eqref{eq:cF}.
Pour $\be\in \{0,1\}^{\indsec}$ on pose
\begin{equation}
\tors_{X,\be}=\tors_X\cap \capu{i\in \indsec,\,e_i=1}\{s_i=0\}.
\end{equation}
\begin{lemme}\label{lm:nbre:pts:tors}
Soit $k$ un corps fini de cardinal $q$. 
Soit $J$ un choix admissible de variables 
et $\be=(\bbf,\bg)\in \{0,1\}^{\indsec}$.
On a
\begin{equation}\label{eq:form:nbre:pts:tors}
\frac
{\# \tors_{X,\be}(k)}
{q^{\,\dim(\tors_{X})}}
=
q^{\,-\abs{\be}}
\,
\left[
1+(q-1)
\prodsub{
j\in J
\\
g_j+\sum_{i\in I_j}f_i=0
}
(1-(1-q^{-1})^{\# I_j})
\right].
\end{equation}
\end{lemme}
\begin{proof}
On peut supposer que $\be=0$ et que $\cup_{j\in J}I_j=I$, les autres
cas s'en déduisent facilement. Soit $(s_i)\in k^I$. Si pour tout
$j\in J$ il existe $i\in I_j$ tel que $s_i=0$ le nombre de solutions 
$(s_j)\in k^J$
de l'équation linéaire $\cF(s_i,s_j)=0$
est $q^{\# J}$. Dans le cas contraire, c'est $q^{\#J-1}$.
En sommant sur tous les $(s_i)\in k^I$, on obtient la relation 
\begin{equation}
\# \tors_X(k)=
q^{\# J}\prod_{j\in J}(q^{\# I_j}-(q-1)^{\# I_j})
+q^{\# J-1}[q^{\# I}-\prod_{j\in J}(q^{\# I_j}-(q-1)^{\# I_j})]
\end{equation}
Compte tenu de $\dim \tors_{X}=\# \indsec-1=\# I+\# J-1$, on en déduit bien la formule \eqref{eq:form:nbre:pts:tors}.
\end{proof}

\subsection{Relèvement du problème de comptage}
On conserve les hypothèses et notations précédentes.
Désormais, le corps de bas $k$ sera toujours soit un corps fini de
cardinal $q$ (avec en vue l'étude de \og Manin I\fg\ et \og Manin
II\fg), soit le corps $\bQ$ des rationnels 
(avec en vue l'étude de \og Manin géométrique\fg~; comme nous l'avons
déjà signalé, la stratégie utilisée est en fait applicable à un
corps de base quelconque mais nous nous limitons au cas des rationnels
par souci de simplicité et parce que les exemples explicites
d'application de notre résultat tombent dans ce cadre).

Soit $J\subset \indsec$ un choix admissible de variables.
On notera alors $I\eqdef \indsec\setminus J$.
\'Ecrivons, pour $j\in J$,
\begin{equation}
\scG_j=\sum_{i\in I} a_{i,j}\,\scF_i,\quad
a_{i,j}\in \bZ.
\end{equation}
Pour $\ecD\in \diveffc$, on note $\scan{\ecD}$ la section canonique de $\OdeC(\ecD)$.

On suppose $k$ fini.
Pour $\becD\in \diveffc^I$ et $\becG\in \diveffc^{J}$,
on d\'esigne par $\fonc(\becD,\becG)$ (respectivement
$\fonc^{\ast}(\becD,\becG)$, respectivement $\fonc_{\jz}(\becD,\becG)$
si $\jz\in J$ le cardinal de l'ensemble
des \'el\'ements $(t_j)_{j\in J}$
(respectivement des \'el\'ements $(t_j)_{j\in J}$ {\em tous non nuls},
respectivement des \'el\'ements $(t_j)_{j\in J}$ avec $t_{\jz}=0$)
du produit
\begin{equation}
\prod_{j\in J} 
H^0(\courbe,\struct{\courbe}
(-\ecG_j+\sum_{i\in I}a_{i,j}\,\ecD_i))
\end{equation}
v\'erifiant la relation
\begin{equation}
\cF(\scan{\ecD_i}\,,t_j\,\scan{\ecG_j})=0.
\end{equation}
Si $k=\bQ$, on définit de même les fonctions $\foncpr$, $\foncpr^{\ast}$
et $\foncjzpr$ en remplaçant dans ce qui précède $\courbe$ par $\cpr$.

Supposons $k$ fini.
Des r\'esultats des sections 1.3, 1.4 et 1.6 de \cite{Bou:compt} d\'ecoulent 
alors, pour tout $y\in\Pic(X)^{\vee}\cap \ceff(X)^{\vee}$,
la formule 
\begin{equation}\label{eq:form:relevement:2}
\# \HOM_{X_0}(\courbe,X,y)(k)
=
\sum_{
\substack{
\becE\in \diveffc^{\indsec}
\\
\deg(\ecG_j) \leq \acc{y}{\scG_j},\quad j\in J
\\
\deg(\ecF_i)\leq \acc{y}{\scF_i},\quad i\in I 
\\
\\
\deg(\ecD_i)= \acc{y}{\scF_i}-\deg(\ecF_i),\quad i\in I 
}
}
\muX(\becG,\becF)
\fonc^{\ast}(\becD+\becF,\becG).
\end{equation}
Pour alléger l'écriture, on adoptera dans toute la suite de cet
article la notation suivante~: soit $f$ une fonction à valeurs
complexes (respectivement $\scP$ une fonction booléenne) définie sur 
$(\Pic(X)^{\vee}\cap \ceff(X)^{\vee})\times \diveffc^{\indsec}$.
On pose alors 
\begin{equation}\label{eq:not:sumyscP}
\sum^y_{\scP}
f
=
\sumsub{
(\becG,\becD)\in \diveffc^{\indsec}
\\
\deg(\ecG_j) \leq \acc{y}{\scG_j},\quad j\in J
\\
\deg(\ecD_i)= \acc{y}{\scF_i},\quad i\in I 
\\
\scP(y,\becG,\becD)
}
f(y,\becG,\becD).
\end{equation}
Ainsi d'après \eqref{eq:form:relevement:2}
et \eqref{eq:def:nuXJ}, on a
\begin{equation}\label{eq:form:relevement}
\# \HOM_{X_0}(\courbe,X,y)(k)
=
\sum^y \nuXJ.\fonc^{\ast}.
\end{equation}
De même, si $k=\bQ$, on adopte pour tous $p,r$ une notation analogue à 
\eqref{eq:not:sumyscP} en remplaçant $\courbe$ par $\cpr$, et on a 
\begin{equation}\label{eq:form:relevement:pr}
\# \HOM_{X_0}(\courbe,X,y)(\fpr)
=
\sum^{y} \nuXJpr.\foncpr^{\ast}.
\end{equation}

\section{\'Etude de quelques  séries génératrices}

\begin{prop}\label{prop:calc:sergen}
Soit $J$ un ensemble fini non vide.
\begin{enumerate}
\item
Soit $I$ un ensemble fini non vide
et
$I=\disju{j\in J}I_j$ une partition de $I$ en ensembles non vides.
Soit $\ba\in \bN^I_{>0}$, $\bnu\in \bN^J$ et 
\begin{equation}
\cF(\rho,\ba,\bnu,\bt)
\eqdef
\sum_{\bd\in \bN^{I}}
\rho^{\,
\Minu{j\in J}(\nu_j+\sum_{i\in I_j}a_i\,d_i)
}
\bt^{\,\bd}
\end{equation}
(cette notation fait abstraction de la partition utilisée, qui sera
toujours clairement indiquée par le contexte).
Soit $\scI$ la classe des parties $K$ de $I$ telles que 
$K\cap I_j$ soit un singleton pour tout $j\in J$.
Pour tout $K\in \scI$ soit $m_{K}$ le plus petit commun multiple
des $(a_{i})_{i\in K}$.
Soit
\begin{equation}
\wt{\cF}(\rho,\ba,\bnu,\bt)
\eqdef
\prod_{K\in \scI}(1-\rho^{\,m_K}\,\produ{i\in K} t_i^{\,\frac{m_K}{a_i}})
\prod_{i\in I}
(1-t_i)\,
\cF(\rho,\ba,\bnu,\bt)
\quad \in \bZ[\rho][[\bt]]
.
\end{equation}
\begin{enumerate}
\item
$\wt{\cF}(\rho,\ba,\bnu,\bt)$ est polyn\^omial en $\rho$ et $\bt$.
\item
On suppose qu'il existe $j_{_{0}}\in J$ tel que $\nu_{\jz}=0$ et $a_{i}=1$
pour tout $i\in I_{\jz}$.
On a alors
\begin{equation}
\deg_{\rho,\bt}\wt{\cF}(\rho,\ba,\bnu,\bt^{-1})\leq 0.
\end{equation}
\item
Pour tout $\eps>0$, il existe $\eta(\eps)>0$
(dépendant de $\ba$ et $\bnu$)
tel qu'on ait pour tout $0\leq \eta\leq \eta(\eps)$
\begin{equation}
\deg_{\rho,\bt}\wt{\cF}(\rho^{\eta},\ba,\bnu,\bt^{-1})\leq \eps.
\end{equation}
\end{enumerate}
\item
Soit $\jz\in J$, $\ba\in \bN^J_{>0}$, $\bnu\in \bN^J$ 
et 
\begin{equation}
\cG(\jz,\rho,\tau,\ba,\bnu,\bt)
\eqdef
\sum_{\bd\in \bN^{J}}
\rho^{\,
\Minu{j\in J}(a_j\,d_j+\nu_j)
}
\tau^{\,
\Minu{j\in J\setminus \{\jz\}}(a_j\,d_j+\nu_j)
}
\bt^{\,\bd}
\quad
\in \bZ[\rho,\tau][[\bt]].
\end{equation}
Soit $m$ le plus petit commun multiple des $(a_j)_{j\in J}$
et $n$ le plus petit commun multiple des $(a_j)_{j\in J\setminus \{\jz\}}$.
Soit
\begin{equation}
\wt{\cG}(\jz,\rho,\tau,\ba,\bnu,\bt)
\eqdef
(1-\rho^{\,m}\,\tau ^m \produ{j\in J} t_j^{\,\frac{m}{a_j}})
(1-\tau ^n \produ{j\in J\setminus \{\jz\}} t_i^{\,\frac{n}{a_j}})
\prod_{j\in J} 
(1-t_j)\,
\cG(\jz,\rho,\tau,\ba,\bnu,\bt).
\end{equation}
\begin{enumerate}
\item
$\wt{\cG}(\jz,\rho,\ba,\bnu,\bt)$ est polyn\^omial en $\rho$ et $\bt$.
\item
On suppose qu'il existe $\ju \in J$ tel que $\nu_{\ju }=0$ et $a_{\ju }=1$.
Pour tout $\eps>0$, 
il existe $\eta(\eps)>0$
(dépendant de $\ba$ et $\bnu$)
tel qu'on ait pour tout $0\leq \eta\leq \eta(\eps)$
\begin{equation}
\deg_{\rho,\bt}\wt{\cG}(\jz,\rho,\rho^{\eta},\ba,\bnu,\bt^{-1})\leq \eps.
\end{equation}
\end{enumerate}
\end{enumerate}
\end{prop}
\begin{proof}
Nous ne démontrons que la partie concernant $\cF$. La partie
concernant $\cG$ se démontre en utilisant la même technique, laquelle
est d'ailleurs une généralisation de celle utilisée dans la démonstration de \cite[proposition 56]{Bou:fam}.

Un calcul \'el\'ementaire montre
que le coefficient $\alpha(\rho,\ba,\bnu,\bd)$ 
de $\bt^{\bd}$ dans
$\wt{\cF}(\rho,\ba,\bnu,\bt)$ 
vaut
\begin{equation}\label{eq:coef:anubd:gen}
\sumsub{
\bgamma\in \{0,1\}^{\scI},\,\,\bmu\in \{0,1\}^I\\~\\
\mu_i+\frac{1}{a_i}\sumsub{K\in \scI\\i\in K}\,\,m_K.\gamma_K\leq d_i,\quad i\in I
}
(-1)^{\sumu{K\in \scI}\gamma_K+\sumu{i\in I} \mu_i}
\rho^{\,\Minu{j\in J}(\nu_j+\sum_{i\in I_j}a_i(d_i-\mu_i))}.
\end{equation}
Soit $\bd\in \bN^I$.
Pour tout $i\in I$, on note $\scP(i)$ l'inégalité
\begin{equation}
d_i\geq 1+\frac{1}{a_i}\sumsub{K'\in \scI\\i\in K'}m_K.
\end{equation}
Supposons tout d'abord qu'il existe $K\in \scI$ tel que tout $i\in K$
vérifie $\scP(i)$. Alors les inégalités définissant le domaine de sommation dans
\eqref{eq:coef:anubd:gen} et faisant intervenir $\gamma_K$ sont
toujours vérifiées~;
on en déduit aussitôt, en sommant par rapport à $\gamma_K$, que $\alpha(\rho,\ba,\bnu,\bd)$ est nul.

On suppose à présent qu'il existe $\jz\in J$ et $\iz \in I_{\jz}$ tels qu'on ait
\begin{equation}\label{eq:min:di0:ji0}
d_{\iz }\geq 1+\frac{1}{a_{\iz }}\left(\sum_{K\in \scI}m_K+\Maxu{j\in J}\abs{\nu_j-\nu_{\jz}}\right).
\end{equation}
Montrons que $\alpha(\rho,\ba,\bnu,\bd)$ est nul, ce qui établira le
caractère polynomial de $\wt{\cF}(\rho,\ba,\bnu,\bt)$.
D'après ce qui précède, on peut supposer que pour tout $K\in \scI$ il
existe $i\in K$ ne vérifiant pas $\scP(i)$.
On en déduit l'existence de $\ju \in J$ tel qu'aucun $i$ de $I_{\ju }$
ne vérifie $\scP(i)$ (si un tel $\ju $ n'existait pas, on pourrait construire facilement un $K\in
\scI$ contredisant la propriété précédente).
On a alors
\begin{equation}
\sum_{i\in I_{\ju }}a_i\,d_i\leq \sum_{K\in \scI} m_K
\end{equation}
et donc d'après \eqref{eq:min:di0:ji0}
\begin{equation}
a_{\iz }\,(d_{\iz }-1)+\nu_{\jz}\geq \sum_{i\in I_{\ju }}a_i\,d_i +\nu_j.
\end{equation}
On en déduit aussitôt qu'on a pour tous les $(\bgamma,\bmu)\in
\{0,1\}^{\scI}\times \{0,1\}^{I}$ vérifiant les inégalités du domaine
de sommation de \eqref{eq:coef:anubd:gen} la relation
\begin{equation}
\sum_{i\in I_{\jz}} a_{i}\,(d_i-\mu_i)+\nu_{\jz}
\geq 
\sum_{i\in I_{\ju }}a_i\,d_i +\nu_{\ju }.
\end{equation}
Ainsi, dans \eqref{eq:coef:anubd:gen}, on peut remplacer $\Minu{j\in J}$
par $\Minu{j\in J,\,j \neq \jz}$ ; par ailleurs, d'après \eqref{eq:min:di0:ji0},
la condition
\begin{equation}
\mu_{\iz }+\frac{1}{a_{\iz }}\sumsub{K\in \scI\\\iz \in K}\,\,m_K.\gamma_K\leq d_{\iz }
\end{equation}
est toujours vérifiée. On en
déduit aussitôt, en sommant par rapport à $\mu_{\iz }$, que  $\alpha(\rho,\ba,\bnu,\bd)$ est nul.

Vu \eqref{eq:coef:anubd:gen}, le deuxième point
découle de l'inégalité évidente
\begin{equation}
\forall \bd\in \bN^I,\quad \nu_{\jz}+\sum_{i\in I_{\jz}}a_i\,d_i-\abs{\bd}\leq 0.
\end{equation}
De même, pour le troisième point, il suffit de
vérifier que si $\eta$ est assez petit on a pour tout $\bd\in \bN^I$
\begin{equation}
\eta \,\,\Minu{j\in J}(\nu_j+\sum_{i\in I_j}a_i\,d_i)-\abs{\bd}\leq \eps.
\end{equation}

\end{proof}

\begin{defi}
Soit $M\in \bZ$.
Un élément $F$ de $\bZ[\rho,\tau][[\bt]]$
est dit $M$-contrôlé s'il s'écrit 
comme une somme finie d'éléments 
de la forme
\begin{equation}
\frac{P(\rho,\tau,\bt)}
{\produ{\alpha\in A}(1-\rho^{\,m_{\alpha}}\,\tau^{\,n_{\alpha}} \produ{i\in I}t_i^{d_{i_{\alpha}}})}
\end{equation}
où $P\in \bZ[\rho,\tau,\bt]$, $A$ est un ensemble fini 
(éventuellement vide)
et pour $\alpha\in A$,  $m_{\alpha}$, $n_{\alpha}$ 
et les $d_{i_{_{\alpha}}}$ sont des entiers positifs,
tels qu'on ait les propriétés suivantes~: 
pour tout $\eps>0$, 
il existe $\eta(\eps)>0$ tel qu'on ait pour 
tout $0\leq \eta(\eps)\leq \eta$
\begin{equation}
\deg_{\rho,\bt}
P(\rho,\rho^{\eta},\bt^{-1})
\leq 
M-2+\eps~;
\end{equation}
et pour tout $\alpha\in A$
\begin{equation}
m_{\alpha}-\sum_{i\in I} d_{i_{_{\alpha}}}
\leq 
-1.
\end{equation}
Si $P\in \bN[\rho,\tau,\bt]$ la première condition est équivalente 
à $\deg_{\rho,\bt} P(\rho,1,\bt^{-1})\leq M-2$.
\end{defi}
Soit $H_{\bullet}=\{H_{\bg}\}_{\bg\in \{0,1\}^I}$ une famille d'éléments de
$\bZ[\rho,\tau][[\bt]]$ telle que $H_{0}(\bt=0)=1$. 
Soit $\eta\geq 0$.
Si $k$ est fini, pour $\becG\in \diveffc^J_{\leq 1}$, on pose
\begin{equation}
\sum_{\bd\in \bN^I} 
a_{\bd,\eta}
(H_{\bullet},\becG)\bt^{\bd}
\eqdef
\prod_{v\in \courbez}
\frac{H_{v(\becG)}(q_v,q_v^{\eta},\bt^{f_v})}
{\prod_{i\in I}(1-t_i^{f_v})}
\end{equation}
et, si $k=\bQ$, pour $\becG\in \diveffcpr^J_{\leq 1}$, on pose
\begin{equation}
\sum_{\bd\in \bN^I} 
a_{\bd,\eta,p,r}(H_{\bullet},\becG)\bt^{\bd}
\eqdef
\prod_{v\in \cprz}
\frac{H_{v(\becG)}(p_v^{\,r},p_v^{\,r\,\eta},\bt^{f_v})}
{\prod_{i\in I}(1-t_i^{f_v})}.
\end{equation}

\begin{lemme}\label{lemme:maj:seriegen}
Avec les notations ci-dessus, on suppose que
$H_{\bg}$ est $\abs{\bg}$-contrôlé si $\bg\in \{0,1\}^J\setminus \{0\}$
et que $H_0-1$  est $0$-contrôlé.

Il existe alors une constante $C>0$ telle que pour tout $\eps>0$ assez
petit, il existe $\eta(\eps)>0$ (tout ceci pouvant être rendu
explicite en fonction de $H_{\bullet}$) tels qu'on ait
pour tout $\bd\in \bN^I$ 
\begin{enumerate}
\item
($k$ fini)
pour tout $\becG\in \diveffc^J_{\leq 1}$
\begin{enumerate}
\item
\begin{equation}
\abs{
a_{\bd,0}(H_{\bullet},\becG)
-
\left(\frac{\hc\,q^{1-\gc}}{q-1}\right)^{\# I}
\prod_{v} 
H_{v(\becG)}(q_v,1,q_v^{-1})
\,q^{\,\abs{\bd}}
}
\ll_{\eps,H_{\bullet}}
\left(
\prod_{v,v(\becG)\neq 0}
C.q_v^{\abs{v(\becG)}-2+\eps}
\right)
\left[
\sum_{i\in I}
q^{\,\abs{\bd}-\eps\,d_i}
\right]
\end{equation}
\item
pour tout $0\leq \eta\leq \eta(\eps)$
\begin{equation}
a_{\bd,\eta}(H_{\bullet},\becG)
\ll_{\eps,H_{\bullet}}
\left(
\prod_{v,v(\becG)\neq 0}
q_v^{\abs{v(\becG)}-2+\eps}
\right)
q^{\,\abs{\bd}}.
\end{equation}
\end{enumerate}
\item
($k=\bQ$)
Il existe $c(p,r)>0$ et $C(p,r,\eps)>0$
vérifiant $\lim_{r\to +\infty}c(p,r)=1$
et 
$\lim_{r\to +\infty}C(p,r,\eps)=1$
tels qu'on a
\begin{enumerate}
\item
\begin{equation}
\abs{
a_{\bd,p,r,0}(H_{\bullet},0)
-
c(p,r)
\,
p^{\,r\,\abs{\bd}}
}
\ll_{H_{\bullet}}
C(p,r,\eps)
\left[
\sum_{i\in I}
p^{\,r\,\abs{\bd}-\eps\,d_i}
\right]
\end{equation}
\item
pour tout $\becG\in \diveffcpr^J_{\leq 1}$
et tout $0\leq \eta\leq \eta(\eps)$
\begin{equation}
\abs{a_{\bd,\eta,p,r}(H_{\bullet},\becG)}
\ll_{H_{\bullet}}
C(p,r,\eps)
\left(
\prod_{v,v(\becG)\neq 0}
C.
p_v^{r\,(\abs{v(\becG)}-2+\eps)}
\right)
p^{\,r\,\abs{\bd}}
\end{equation}
\end{enumerate}
\end{enumerate}
\end{lemme}
\begin{proof}
Elle est contenue
dans \cite[p.18]{Bou:fam} si $k$ est fini et dans
\cite[Section 6]{Bou:moduli} si $k=\bQ$. Indiquons seulement les
grandes lignes~: si $k$ est fini, on applique
\cite[Lemma 5.1]{Bou:moduli} (qui est essentiellement un avatar des
estimations de Cauchy) à la série
\begin{equation}
\cH(\becG,\bt)\eqdef
\prod (1-q\,t_i)
\sum_{\bd\in \bN^I}a_{\bd,\eta}(G,\becG)\bt^{\bd}
=\cH(0,\bt)
\prod_{v,v(\becG)\neq 0}
\frac{H_{v(\becG)}(q_v,q_v^{\eta},\bt^{f_v})}{H_0(q_v,q_v^{\eta},\bt^{f_v})}.
\end{equation}
Les hypothèses faites sur $H_{\bullet}$ permettent de majorer 
de manière ad hoc la norme sup
\begin{equation}
\norm{\produ{v,v(\becG)\neq 0}
\frac{H_{v(\becG)}(q_v,q_v^{\eta},\bt^{f_v})}{H_0(q_v,q_v^{\eta},\bt^{f_v})}
}_{q^{-1+\eps}}.
\end{equation}

Si $k=\bQ$ le raisonnement est analogue. 
Il faut montrer en outre dans ce cas qu'on a
\begin{equation}
\lim_{r\to +\infty}\wt{\cH}(0,(p^{-r},\dots,p^{-r}))=1
\end{equation}
et, pour $\eps$ assez petit,
\begin{equation}
\limsup_{r\to +\infty}\norm{\wt{\cH}(0,\bt)}_{p^{r(-1+\eps)}}=1
\end{equation}
où $\wt{\cH}(0,\bt)\eqdef \cH(0,\bt)\prod (1-t_i)$
ce qui se fait en utilisant la décomposition en produit eulérien et en
appliquant \cite[Lemma 5.2]{Bou:moduli}.
\end{proof}

Désormais, on considère une hypersurface intrinsèque linéaire $X$ et
un ensemble admissible de variables $J$ (\cf la sous-section \ref{subsec:hypintlin}). 
Rappelons qu'on a alors $\dim(X)+1=\# J$ et $\#I=\rk(\Pic(X))$. 
Dans tout ce qui suit, si $A$, $B$ et $x$ sont des quantités, la notation
$A\ll_{x}B$ signifiera \og $A$ est majoré par $B$ à une constante
multiplicative près ne dépendant que de $X$, $\courbe$ et $x$\fg. Les
constantes introduites ne dépendront que de $X$ et $\courbe$. 

Nous aurons besoin de l'hypothèse suivante sur $X$.
\begin{hyp}\label{hyp:noninter}
\begin{enumerate}
\item\label{item:1:hyp:noninter}
Au plus $\dim(X)+1$ diviseurs parmi les $\{\cD_i\}_{i\in \classadminc}$ s'intersectent mutuellement.
En d'autres termes, on a $\Max\{\# K, \,\,K\in \classadminc\}\leq \dim X+1$.
\item\label{item:2:hyp:noninter}
si $K\in \classadminc$ est tel que $K\cap I_j$ est un
singleton pour tout $i$, alors il existe $i\in K$ tel que 
$a_i=1$.
\end{enumerate}
\end{hyp}
\begin{ex}
On vérifie aussitôt que cette hypothèse est satisfaite pour les
surfaces de la liste de \cite{Der:sdp:ut:hyp}, ainsi que pour la
famille construite dans \cite{Bou:fam}. Nous ne savons pas s'il existe
des exemples d'hypersurfaces linéaires intrinsèques pour lesquelles
elle n'est pas vérifiée.
\end{ex}
Supposons $k$ fini.
Pour $\becG\in \diveffc^J$ et $\bd\in \bN^I$, 
on pose
\begin{equation}\label{eq:def:scn}
\scM(\bd,\becG)
\eqdef
\sumsub{
\becD\in \diveffc^I
\\
\deg(\becD)=\bd
}
\nuXJ(\becG,\becD)
\,
q^{
\deg(
\pgcd_{j\in J}(
\scan{\cG_j}
\prod_{i\in I_j}\scan{\ecD_i}^{a_{i}})
)
}
\end{equation}
et, pour $\eta\geq 0$ et $\jz\in J$, 
\begin{equation}\label{eq:defjzeta}
\scM_{_{\jz,\eta }}(\bd,\becG)
\eqdef
\sumsub{
\becD\in \diveffc^I
\\
\deg(\becD)=\bd
}
\abs{\nuXJ(\becG,\becD)}
q^{\deg(\pgcd_{j\in J}(\scan{\cG_j}\prod_{i\in I_j}\scan{\ecD_i}^{a_{i}}))
+\eta\,
\deg(\pgcd_{j\in J\setminus \{\jz\}}(\scan{\cG_j}\prod_{i\in I_j}\scan{\ecD_i}^{a_{i}}))}.
\end{equation}
Notons que d'après la remarque \ref{rem:nuXJnul}, ces quantités sont
nulles si $\becG\notin \diveffc^J_{\leq 1}$.

Si $k=\bQ$, on définit de même les fonctions $\scM_{_{p,r}}$ et $\scM_{_{\jz,\eta,p,r }}$
sur $\bN^I\times \diveffcpr^J$ en remplaçant dans les formules ci-dessus
$\courbe$ par $\cpr$, $\nuXJ$ par $\nuXJpr$ et $q$ par $p^r$.

Posons $\scF(n)=\Min(1,n)$.
Pour $\bg\in \bN^J$, soit
\begin{equation}
F_{\bg}(\rho,\bt)
\eqdef
\sum_{\bbf\in \bN^I}
\nuzXJ(\bg,\bbf)
\rho^{\,
\Minu{j\in J}
(g_j+\sumu{i\in I_j} a_{i}f_i)}
\bt^{\bbf},
\end{equation}
\begin{equation}
F_{1,\bg}(\rho,\bt)
\eqdef
\sum_{\bbf\in \bN^I}
\nuzXJ(\bg,\bbf)
\rho^{\,\scF\left(
\Minu{j\in J}(g_j+\sumu{i\in I_j} a_{i}f_i)
\right)
}
\bt^{\bbf},
\end{equation}
\begin{equation}
F_{2,\bg}(\rho,\bt)
\eqdef
F_{\bg}(\rho,\bt)-F_{1,\bg}(\rho,\bt)
\end{equation}
et, pour $\alpha\in \{1,2\}$,
\begin{equation}
H_{\alpha,\bg}(\rho,\bt)
\eqdef
\left(\prod_{i\in I}1-t_i\right)
F_{\alpha,\bg}(\rho,\bt).
\end{equation}
Pour  $\jz\in J$ on pose
\begin{equation}
F_{\jz,\bg}(\rho,\tau,\bt)
\eqdef
\sum_{\bbf\in \bN^I}
\abs{\nuzXJ(\bg,\bbf)}
\rho^{
\Minu{j\in J}(g_j+\sum_{i\in I_j} a_{i}f_i)
}
\tau^{
\Minu{j\in J\setminus \{\jz\}}(g_j+\sum_{i\in I_j} a_{i}f_i)
}
\bt^{\bbf},
\end{equation}
\begin{equation}
F_{\jz,1,\bg}(\rho,\tau,\bt)
\eqdef
\sum_{\bbf\in \bN^I}
\abs{\nuzXJ(\bg,\bbf)}
\rho^{
\scF\left(\Minu{j\in J}(g_j+\sum_{i\in I_j} a_{i}f_i)\right)
}
\tau^{
\scF\left(\Minu{j\in J\setminus \{\jz\}}(g_j+\sum_{i\in I_j}
  a_{i}f_i)\right)
}
\bt^{\bbf}
\end{equation}
et
\begin{equation}
F_{\jz,2,\bg}
\eqdef
F_{\jz,\bg}
-
F_{\jz,1,\bg}.
\end{equation}

\begin{prop}\label{prop:res:sergen}
On suppose l'hypothèse \ref{hyp:noninter} vérifiée.
\begin{enumerate}
\item
Pour tout $\bg\in \{0,1\}^J$, $H_{1,\bg}$ est un élément de $\bZ[\rho,\bt]$ .
$H_{1,0}-1$ est $0$-contrôlée et $H_{1,\bg}$ est $\abs{\bg}$-contrôlée
pour $\bg\neq 0$
\item
Pour tout $\bg\in \{0,1\}^J$, $H_{2,\bg}$ est 
$\abs{\bg}$-contrôlée.
\item
Il existe un élément $H_{\jz,1,\bg}\in \bN[\rho,\tau,\bt]$
tel que $H_{\jz,1,0}-1$ est $0$-contrôlé, $H_{\jz,1,\bg}$ est $\abs{\bg}$-contrôlée
pour $\bg\neq 0$, et pour tout $\rho,\tau>0$  
\begin{equation}
F_{\jz,1,\bg}(\rho,\tau,\bt)\left(\prod_{i\in I} 1-t_i\right)
\leq H_{\jz,1,\bg}(\rho,\tau,\bt).
\end{equation}
\item
Pour tout $\bg\in \{0,1\}^J$, $F_{\jz,2,\bg}$ est $\abs{\bg}$-contrôlé.
\item\label{item:prop:res:sergen:kfini}
($k$ fini)
Pour tout $\becG\in \diveffc^J_{\leq 1}$, 
posons
\begin{equation}
\cprinc(\becG)
\eqdef 
\left(
\frac{\hc\,q^{1-\gc}}{q-1}
\right)^{\# I}
\prod_{v\in \courbez} 
H_{v(\becG)}(q_v,q_v^{-1}).
\end{equation}
Il existe alors $C>0$ telle que pour tout $\eps>0$ assez petit,
il existe $\eta(\eps)>0$ (tout ceci pouvant être explicitement décrit en fonction de $X$)
tels qu'on ait pour tout $\bd\in \bN^I$,
et tout $\becG\in \diveffc^J_{\leq 1}$ 
\begin{equation}\label{eq:maj:NdG}
\abs{
q^{\,-\abs{\bd}}
\scM(\bd,\becG)
-
\cprinc(\becG)
}
\ll_{\eps}
\left(
\prod_{v,v(\becG)\neq 0}
C.
q_v^{\abs{v(\becG)}-2+\eps}
\right)
\left[
\sum_{i\in I}
q^{\,-\eps\,d_i}
\right]
\end{equation}
et pour tout $0\leq \eta\leq \eta(\eps)$
\begin{equation}\label{eq:maj:ceps}
q^{-\abs{\bd}}\abs{\scM_{\jz,\eta }(\bd,\becG)}
\ll_{\eps}
\left(
\prod_{v,v(\becG)\neq 0}
C.
q_v^{\abs{v(\becG)}-2+\eps}
\right)
\end{equation}
Par ailleurs, on a
\begin{equation}\label{eq:sum:becG:gamma}
q^{\,(1-\gc)\dim(X)}
\sum_{\becG\in 
\diveffc_{\leq 1}^J}
\cprinc(\becG)
\,
q^{-\abs{\deg(\becG)}}=\gamma(X)
\end{equation}
et pour tout $\eps>0$ assez petit
\begin{equation}\label{eq:sum:becG:gamma:eps}
\abs{
\gamma(X)
-
q^{\,(1-\gc)\dim(X)}
\sumsub{\becG\in \diveffc_{\leq 1}^J
\\
\deg(\ecG_j)\leq \acc{y}{\scG_j}, \quad j\in J
}
\cprinc(\becG)
\,
q^{-\abs{\deg(\becG)}}
}
\ll_{\eps} 
q^{-\eps\sum_{j\in J}\acc{y}{\scG_j}}.
\end{equation}

\item\label{item:prop:res:sergen:kQ}
($k=\bQ$)

Il existe $C>0$ et $c(p,r)>0$  telles que pour tout $\eps>0$ assez
petit il existe $C(p,r,\eps)>0$ et $\eta(\eps)> 0$ (tout ceci
pouvant être rendu explicite en fonction de $X$), 
avec 
$\lim_{r\to +\infty}c(p,r)=1$
et 
$\lim_{r\to +\infty}C(p,r,\eps)=1$
tels qu'on ait pour tout $\bd\in \bN^I$,
pour tout $\becG\in \diveffcpr^J_{\leq 1}$ et tout $0\leq \eta\leq \eta(\eps)$
\begin{equation}\label{eq:maj:Nd0:pr}
\abs{
p^{-r\,\abs{\bd}}
\scM_{_{p,r}}(\bd,0)
-
c(p,r)\,p^{\,r\,\abs{\bd}}
}
\leq
C(p,r,\eps)
\left[
\sum_{i\in I}
p^{\,-\eps\,d_i)}
\right]
\end{equation}
\begin{equation}\label{eq:maj:ceps:pr}
\text{et}\quad
p^{-r\,\abs{\bd}}
\scM_{\jz,\eta,p,r }(\bd,\becG)
\leq
C(p,r,\eps)
\left(
\prod_{v,v(\becG)\neq 0}
C.
p_v^{\,r(\abs{v(\becG)}-2+\eps)}
\right).
\end{equation}
\end{enumerate}
\end{prop}
\begin{proof}
Compte tenu de \eqref{eq:def:nu0}, on a
\begin{equation}
F_{1,\bg}(\rho,\bt)
=
\sumsub{
\bm\in \{0,1\}^I
\\
\bbf\in \bN^I}
\muzX (\bg,\bm)
\bt^{\bm+\bbf}
+
\sumsub{
\bm\in \{0,1\}^I,\,\bbf\in \bN^I
\\
\forall j\in J,\quad g_j+\sum_{i\in I_j}(m_i+f_i)\geq 1
}
\muzX (\bg,\bm)
(\rho-1)
\bt^{\bm+\bbf}.
\end{equation}
Or pour  $\bm\in \{0,1\}^I$
on a
\begin{align}
\sumsub{
\bbf\in \bN^I
\\
\forall j\in J,\quad g_j+\sum_{i\in I_j}(m_i+f_i)\geq 1
}
\bt^{\bbf}
&=
\sumsub{
\bbf\in \bN^I
\\
\forall j\in J,\quad (g_j+\sum_{i\in I_j}m_i=0)\imply \sum_{i\in I_j}f_i\geq 1
}
\bt^{\bbf}
\\
&
=
\frac{1}
{\produ{i\in I}(1-t_i)}
\times
\prodsub{
j\in J
\\
g_j+\sum_{i\in I_j}m_i=0
}
(1-\prod_{i\in I_j}(1-t_i)).
\end{align}
On en déduit l'égalité
\begin{equation}\label{eq:expr:G}
H_{1,\bg}(\rho,\bt)
=
\sum_{\bm\in \{0,1\}^I}
\muzX (\bg,\bm)
\,
\bt^{\bm}
\,
\left[
1+(\rho-1)
\prodsub{
j\in J
\\
g_j+\sum_{i\in I_j}m_i=0
}
(1-\prod_{i\in I_j}(1-t_i))
\right].
\end{equation}
Mais pour tous $\bm$, $\bg$ on a 
\begin{equation}
\sum_{i\in I}m_i+\#\{j\in J,\,\,g_j+\sum_{i\in I_j}m_j=0\}
\geq \#\{j\in J, \,\,g_j=0\}=\#J-\abs{\bg}\geq 3-\abs{\bg}
\end{equation}
On déduit de l'inégalité précédente et de \eqref{eq:expr:G} 
les majorations
\begin{equation}
\deg_{\rho,\bt}H_{1,\bg}(\rho,\bt^{-1})
\leq 
\Max(
\abs{\bg}-2
,
\Maxu{
\substack{
\bm\in \{0,1\}^I
\\
\muzX (\bm,\bg)\neq 0
}
} 
(-\abs{\bm})
)
\end{equation}
et 
\begin{equation}
\deg_{\rho,\bt}[H_{1,0}(\rho,\bt^{-1})-1]
\leq 
\Max(-2
,
\Maxu{
\substack{
\bm\in \{0,1\}^I
\\
\abs{\bm}\neq 0
\\
\muzX (\bm,0)\neq 0
}
} 
(-\abs{\bm})
)
\end{equation}
Soit $\bm\in \{0,1\}^I$ et $\bg\in \{0,1\}^J$ tels que 
$\muzX (\bg,\bm)\neq 0$ et tels qu'en outre 
si $\abs{\bg}=0$ on a
${\abs{\bm}}\neq 0$. 
On a alors
\begin{equation}
-\abs{\bm}\leq \abs{\bg}-2.
\end{equation}
C'est en effet immédiat si $\abs{\bg}\geq 2$ et si $\abs{\bg}=1$ ou
$\abs{\bg}=0$ cela découle de la proposition \ref{prop:mu}.
Ceci conclut la démonstration du premier point.

Passons au deuxième point.
La remarque \ref{rem:nu0nonnul}
permet d'écrire 
\begin{equation}\label{eq:F2grhobt}
F_{2,\bg}(\rho,\bt)
=
\sumsub{
\bbf\in \bN^I
\\
(0,\bbf)\in \Nindprim 
\\
\Minu{j\in J}(g_j+\sum_{i\in I_j} a_{i}f_i)\geq 2
}
\nuzXJ (\bg,\bbf)
(\rho^{\Minu{j\in J}(g_j+\sum_{i\in I_j} a_{i}f_j)}
-\rho)
\bt^{\bbf}
\end{equation}
Soit $\scI$ la classe des sous-ensembles $K\subset I$ de cardinal $\#J$
tel que $K\in \classadminc$ 
et $K\cap I_j$ est un singleton pour tout $j\in J$.
Compte tenu du point \ref{item:1:hyp:noninter} de l'hypothèse \ref{hyp:noninter}, 
on voit alors que si $\bbf\in \bN^I$ 
vérifie les conditions sous le signe somme dans \eqref{eq:F2grhobt},
il existe $K\in \scI$  
tel que $f_i\geq 1$ si et seulement si $i\in K$. On a alors
(\cf la remarque \ref{rem:sum:bbfK:n0})
$\nuzXJ (\bg,\bbf)=\nuzXJ (\bg,\bbf_{K})$. 
La fonction $F_{2,\bg}$  se réécrit donc
comme une somme sur tous les
$K\in \scI$ des fonctions 
\begin{equation}\label{eq:expr:wtI}
\nuzXJ (\bg,\bbf_{K})
\sumsub{
\bbf\in \bN_{>0}^J
\\
\Minu{j\in J}(g_j+a_{j}f_j)\geq 2
}
(\rho^{\Minu{j\in J}(g_j+a_{j}f_j)}
-\rho)
\bt^{\bbf}.
\end{equation}
Pour alléger l'écriture, on a fixé dans l'expression précédente une identification de $K$ à $J$.
Soit $J_1\eqdef\{j\in J, a_j=1\}$.
D'après le point \ref{item:2:hyp:noninter} de l'hypothèse \ref{hyp:noninter}, $J_1$ est non vide. 
Notons $m$ le plus petit commun multiple de la famille $(a_j)_{j\in J}$.
L'expression \eqref{eq:expr:wtI} se réécrit alors, 
compte tenu de la proposition
\ref{prop:calc:sergen},
\begin{equation}
\nuzXJ (\bg,\bbf_{K})
\sumsub{
\bbf\in \bN_{>0}^K
\\
\forall j\in J_1,\,\,g_j=0\imply f_j\geq 2
}
(\rho^{\Minu{j\in J}(g_j+a_j\,f_j)}
-\rho)
\bt^{\bbf}
\end{equation}
\begin{equation}
=
\nuzXJ (\bg,\bbf_{K})
\prod_{j\in J_1} t_j^{2-g_j}
\prod_{j\in J\setminus J_1} t_j
\sumsub{
\bbf\in \bN_{\geq 0}^J
}
(\rho^{2+\Min((f_j)_{j\in J_1},(g_j+a_j-2+f_j)_{j\in J\setminus J_1})}
-\rho)
\bt^{\bbf}
\end{equation}
\begin{equation}\label{eq:expr:G2}
=
\nuzXJ (\bg,\bbf_{K})
\rho
\prod_{j\in J_1} t_j^{2-g_j}
\prod_{j\in J\setminus J_1} t_j
\left[
\frac{
\rho\,
\wt{\cF}(\rho,(a_j)_{j\in J},\bnu,\bt)}{
(1-\rho^m\,\produ{j\in J}  t_j^{\frac{m}{a_j}})
\produ{i\in J} (1-t_j)}
-
\frac{1}{\produ{i\in J} 1-t_j}
\right]
\end{equation}
avec $\bnu=((0,\dots,0)_{j\in J_1},(g_j+a_j-2)_{j\in J\setminus J_1})$.
Remarquons que si l'on fixe tous les paramètres sauf $(g_j)\in \{0,1\}^{J_1}$,
l'expression précédente vaut à une constante près 
$\nuzXJ ((g_j)_{j\in J_1},(g_j)_{j\in J\setminus J_1},\bbf_{K})\prod_{j\in J_1}t_j^{-g_j}$.

D'après la proposition \ref{prop:calc:sergen}, on a 
\begin{equation}
\deg_{\rho,\bt}(
\wt{\cF}
(\rho,(a_j)_{j\in  J},\bnu,\bt^{-1})
)
\leq 0.
\end{equation}
On a également
\begin{equation}
\deg_{\rho,\bt}\left(
\rho
\prod_{j\in J_1} t_j^{g_j-2}
\prod_{j\in J\setminus J_1} t_j^{-1}
\right)
=1-\#J_1-\# J +\sum_{j\in J_1} g_j \leq \abs{\bg}-3
\end{equation}
(rappelons qu'on a $\# J_1\geq 1$ et $\# J\geq 3$).

Finalement, compte tenu du fait que $J_1$ est non vide, on a
\begin{equation}
m-\sum_{i\in J} \frac{m}{a_j}\leq -(\#J_1-1)-(\#J-\# J_1)\leq -2.
\end{equation}
Ceci conclut la démonstration du deuxième point.

Pour le troisième point, 
on a la majoration
\begin{equation} 
F_{\jz,1,\bg}(\rho,\bt)
\leq
\sum_{\bbf\in \bN^I}
\sum_{0\leq \bbf'\leq \bbf}\abs{\muzX (\bg,\bbf')}
\rho^{
\scF(\Minu{j\in J}(g_j+\sum_{i\in I_j} a_{i}f_i))
}
\tau^{
\scF(
\Minu{j\in J\setminus \{\jz\}}(g_j+\sum_{i\in I_j} a_{i}f_i)
)
}
\bt^{\bbf}.
\end{equation}
Notons $\frac{H_{\jz,1,\bg}(\rho,\bt)}{\prod_{i\in I}(1-t_i)}$ cette dernière expression.
Un calcul similaire à celui mené pour le premier point montre qu'on a 
\begin{equation}
H_{\jz,1,\bg}(\rho,\bt)
=
\sum_{\bm\in \{0,1\}^I}
\abs{\muzX (\bg,\bm)}
\,
\bt^{\bm}
\,
H_{\bm,\jz,1,\bg}(\rho,\tau,\bt)
\end{equation}
avec, pour tous $\bg$, $\bm$,
\begin{equation}
H_{\bm,\jz,1,\bg}(\rho,\tau,\bt)
=
1+(\rho\,\tau-1)
\prodsub{
j\in J
\\
g_j+\sum_{i\in I_j}m_i=0
}
(1-\prod_{i\in I_j}(1-t_i))
\end{equation}
si $g_{\jz}\neq 0$ ou s'il existe $i\in I_{\jz}$ tel que $m_i\neq 0$,
et
\begin{equation}
H_{\bm,\jz,1,\bg}(\rho,\tau,\bt)
=
1+
(\rho\,\tau-1)
\prodsub{
j\in J
\\
g_j+\sum_{i\in I_j}m_i=0
}
(1-\prod_{i\in I_j}(1-t_i))
+
\tau
\prod_{i\in I_{\jz}}(1-t_i)
\prodsub{
j\in J\setminus \{\jz\}
\\
g_j+\sum_{i\in I_j}m_i=0
}
(1-\prod_{i\in I_j}(1-t_i))
\end{equation}
si $g_{\jz}=0$ et pour tout $i\in I_{\jz}$, $m_i=0$.
On conclut de manière similaire au premier point.

La démonstration du quatrième point est reportée dans un appendice.
L'idée est similaire à celle utilisée pour le deuxième point, avec
quelques complications techniques.

Les deux derniers points se déduisent aussitôt des quatre premiers par
application du lemme \ref{lemme:maj:seriegen}, à l'exception de \eqref{eq:sum:becG:gamma}
et \eqref{eq:sum:becG:gamma:eps}.
Or on a 
\begin{equation}
\sum_{\becG\in \diveffc_{\leq 1}^J}
\prod_{v} H_{v(\becG)}(q_v,q_v^{-1})
q^{-\abs{\deg (\becG)}}
\\
=
\prod_v
\sum_{\bg\in \{0,1\}^{J}}
H_{\bg}(q_v,q_v^{-1})
q^{-{\abs{\bg}}}.
\end{equation}
Tout d'abord, d'après la remarque qui suit \eqref{eq:expr:G2} et la remarque
\ref{rem:sum:bbfK:n0}, on a 
\begin{equation}
\sum_{\bg\in \{0,1\}^{J}}H_{2,\bg}(q_v,q_v^{-1})q^{-{\abs{\bg}}}=0.
\end{equation}
Par ailleurs, d'après \eqref{eq:expr:G}, on a
\begin{multline}
\sum_{\bg\in \{0,1\}^{J}}
H_{1,\bg}(q_v,q_v^{-1})
q^{-{\abs{\bg}}}
=
\sum_{\be=(\bg,\bbf)\in \{0,1\}^{I\cup J}}
\muzX (\be)
\,
q_v^{\,-\abs{\be}}
\,
\left[
1+(q_v-1)
\prodsub{
j\in J
\\
g_j+\sum_{i\in I_j}f_i=0
}
(1-(1-q_v^{-1})^{\# I_j})
\right]
\end{multline}
D'après le lemme \ref{lm:nbre:pts:tors}
et \cite[Lemme 1.25]{Bou:compt}, on a donc
\begin{equation}
\sum_{\bg\{0,1\}^{J}}
H_{\bg}(q_v,q_v^{-1})
q^{-{\abs{\bg}}}
=
\sum_{\be=(\bg,\bbf)\in \{0,1\}^{I\cup J}}
\muzX (\be)
\,
\frac{\#\tors_{X,\be}(\kappa_v)}
{q_v^{\dim(\tors_X)}}
=
(1-q_v^{-1})^{\rg(\Pic(X))}
\frac{\# {X(\kappa_v)}}
{q_v^{\dim(X)}}
\end{equation}
d'où \eqref{eq:sum:becG:gamma}~; \eqref{eq:sum:becG:gamma:eps}
en découle facilement.
\end{proof}

\section{Comptage de sections globales}
On considère une hypersurface intrinsèque linéaire $X$ et
un ensemble admissible de variables $J$ (\cf la sous-section \ref{subsec:hypintlin}). 
On suppose en outre que $\dim(X)=2$ (en d'autres termes que $\#J=3$).
Pour $\jz\in J$, $y\in \Pic(X)^{\vee}$ et $(\becG,\becD)\in
\diveff(\courbe)^{J\times I}$, on pose
\begin{equation}\label{eq:phijz}
\varphi_{\jz}(y,\becG,\becD)
\eqdef
\acc{y}{\sum_{j\neq \jz}\scG_j-\Dtot}-\sum_{j\neq \jz}\deg(\ecG_j)
+
\deg \pgcdu{j\in  \{2,3\}}\left(\ecG_j+\sum_{i\in I_j} b_{i,j}\ecD_i\right),
\end{equation}
\begin{equation}\label{eq:psijz}
\psi_{\jz}(y,\becG,\becD)
\eqdef
\acc{y}{\scG_{\jz}}
-\deg(\ecG_{\jz})
+
\deg \pgcdu{j\in J}\left(\ecG_j+\sum_{i\in I_j} b_{i,j}\ecD_i\right)
-\deg \pgcdu{j\in  J\setminus \{\jz\}}\left(\ecG_j+\sum_{i\in I_j} b_{i,j}\ecD_i\right)
\end{equation}
\begin{equation}\label{eq:theta}
\text{et}\quad
\Theta(y,\becG,\becD)
\eqdef
\varphi_{\jz}(y,\becG,\becD)+\psi_{\jz}(y,\becG,\becD).
\end{equation}
On notera que $\Theta$ ne dépend pas du choix de $\jz$.
Le lemme suivant découle facilement de \cite[Lemme 3.5 et Corollaire
3.6]{Bou:compt}.
\begin{lemme}\label{lm:compt}
Soit $y\in\Pic(X)^{\vee}\cap \ceff(X)^{\vee}$, 
et $(\becG,\becD)\in \diveff(\courbe)^{J\times I}$
vérifiant $\deg(\ecD_i)=\acc{y}{\scF_i}$ pour $i\in I$.
\begin{enumerate}
\item\label{item:lm:compt:1}
Supposons $\varphi_{\jz}(y,\becG,\becD)\geq 2\,\gc-1$ et 
$\psi_{\jz}(y,\becG,\becD)\geq 2\,\gc-1$. Alors
\begin{equation}
\fonc(\becG,\becD)
=
q^{(1-\gc)\dim(X)+\Theta(y,\becG,\becD)}
\end{equation}
et
\begin{equation}
\fonc_{\jz}(\becG,\becD)
=
q^{(1-\gc)\dim(X)+\varphi_{\jz}(y,\becG,\becD)}.
\end{equation}
\item\label{item:lm:compt:2}
Supposons $\varphi_{\jz}(y,\becG,\becD)\geq 0$
Alors
\begin{equation}
\fonc_{\jz}(\becG,\becD)
\leq 
q^{\dim(X)+\varphi_{\jz}(y,\becG,\becD)}.
\end{equation}
\item
Supposons $\psi_{\jz}(y,\becG,\becD)<0$. 
Alors
\begin{equation}
\fonc^{\ast}(\becG,\becD)=0.
\end{equation}
\item\label{item:lm:compt:avant:dern}
Supposons $\varphi_{\jz}(y,\becG,\becD)\leq 2\,\gc-2$. 
Alors
\begin{equation}
\fonc^{\ast}(\becG,\becD)\leq q^{\dim(X)+2\,\gc-2+\psi_{\jz}(\becG,\becD)}.
\end{equation}
\item\label{item:lm:compt:dern}
Supposons $\varphi_{\jz}(y,\becG,\becD)\geq 0$ et 
$\psi_{\jz}(y,\becG,\becD)\geq 0$.
Alors 
\begin{equation}
\fonc^{\ast}(\becG,\becD)\leq q^{\dim(X)+\Theta(y,\becG,\becD)}.
\end{equation}
\end{enumerate}
\end{lemme}

\section{Le résultat principal}\label{sec:res:main}
Supposons $k$ fini. 
Soit $f\,:\,\Pic(X)^{\vee}\cap \ceff(X)^{\vee}\to \bR$ une
fonction. Soit $\cP$ une partie de $\ceff(X)^{\vee}$ qui est une union
finie de cônes.
Pour  $\alpha\in \{I,II\}$,
on dit que $f$ est un terme d'erreur 
de type $\alpha$ sur $\cP$ si  
\begin{itemize}
\item
cas $\alpha=I$~: la série
\begin{equation}
\sum_{y\in \cP\cap \Pic(X)^{\vee}}\,
f(y)
\,
t^{\,\acc{y}{-\can{X}}}
\end{equation}
est $q^{-1}$-contr\^ol\'ee \`a l'ordre $\rg(\Pic(X))-1$.
\item
cas $\alpha=II$:
on a 
\begin{equation}
\lim_{
\substack{
y\in \cP\cap \Pic(X)^{\vee}
\\
\dist(y,\partial \cP)\to +\infty
}
}
q^{\,\acc{y}{\can{X}}}\,f(y)=0.
\end{equation}
\end{itemize}
On dit que $f$ est un terme d'erreur 
de type $\alpha\in \{I,II\}$ si $f$ est un terme d'erreur 
de type $\alpha$ sur $\ceff(X)^{\vee}$. 

Pour $\alpha\in \{I,II\}$, on dit que $f$ est un terme dominant de type $\alpha$ 
sur $\cP$ (respectivement un terme dominant de type $\alpha$)
si $f(y)-q^{\acc{y}{-\can{X}}}\gamma(X)$ est un terme d'erreur de type
  $\alpha$  sur $\cP$ (respectivement un terme d'erreur de type $\alpha$).
On emploiera l'expression \og de type I fort\fg\ quand la propriété
correspondante vaut sur toute partie $\cP'$ de $\cP$ qui est une union
finie de cônes.
\begin{rem}\label{rem:pte:plus:forte}
Aucune des deux propriétés \og $f$ est un terme d'erreur de type
$\alpha$ sur $\cP$\fg\ pour $\alpha\in \{I\text{ fort},II\}$ n'implique a priori l'autre.
Cependant, en supposant que $\cP$ est un cône et que 
 $\{x_0,\dots,x_r\}$ sont des éléments non nuls de son dual, elles
 découlent toutes les deux de la propriété plus forte suivante~:
\begin{equation}
\exists \eps>0,\quad \forall y\in \cP\cap \Pic(X)^{\vee},\quad 
\abs{q^{\acc{y}{\can{X}}}\,f(y)}
\ll \sum_{i=0}^r q^{-\eps \acc{y}{x_i}}.
 \end{equation}
C'est immédiat pour $\alpha=II$
et pour $\alpha=I\text{ fort}$, on utilise le point 1 de \cite[Lemme 13]{Bou:fam}.
\end{rem}

Supposons $k=\bQ$. 
Soit, pour presque tout premier $p$, 
$f_{p,r}\,:\,\Pic(X)^{\vee}\cap \ceff(X)^{\vee}\to \bR$ 
une famille de fonction indexées par $r\in \bN_{>0}$.
Soit $\cP$ une partie de $\ceff(X)^{\vee}$.
On dit que $\{f_{p,r}\}$ est un terme d'erreur sur $\cP$ s'il existe
 une constante positive $M$ dépendant de $\courbe$, $X$, et $\cP$
telle que pour tout  $y\in \ceff(X)^{\vee}\cap \Pic(X)^{\vee}$ 
vérifiant $\dist(y,\partial \cP)\geq M$
on a pour presque tout $p$
\begin{equation}
\lim_{
\substack{
r\to +\infty
}
}
p^{\,r[(\gc-1)\dim(X)+\acc{y}{\can{X}}]}\,f_{p,r}(y)
=0.
\end{equation}
On dit que $\{f_{p,r}\}$ est un terme d'erreur 
si $\{f_{p,r}\}$ est un terme d'erreur sur $\ceff(X)^{\vee}$. 
On dit que $\{f_{p,r}\}$ est un terme dominant sur $\cP$
(respectivement un terme dominant)
si, pour presque tout $p$,
$
f_{p,r}(y)-p^{\,r[(1-\gc)\dim(X)+\acc{y}{-\can{X}}]}
$
est un terme d'erreur sur $\cP$ (respectivement un terme d'erreur).

~

On considère toujours une hypersurface intrinsèque linéaire $X$ telle que
$\dim(X)=2$ et un ensemble admissible de variables $J$. 
Soit $y\in \Pic(X)^{\vee}\cap \ceff(X)^{\vee}$.
Si $k$ est fini, on pose
\begin{equation}\label{eq:defn0}
n^J_0(y)
\eqdef
\sum^y
\nuXJ.q^{(1-\gc)\dim(X)+\Theta}
\end{equation}
\begin{multline}
n^J_1(y)\eqdef \sum^y _{\exists j\in J,\quad
\psi_j<0} 
\nuXJ.q^{(1-\gc)\dim(X)+\Theta}
\\
+
\sumsuby{
\exists j\in J,\quad \varphi_j\geq 2\,\gc-1\\
\forall j\in J,\quad \psi_j\geq 0} 
\nuXJ.(q^{(1-\gc)\dim(X)+\Theta}-\fonc^{\ast})
\end{multline}
et
\begin{equation}
n^J_2(y)
\eqdef
\sumsuby{
\forall j\in J,\quad \varphi_j\leq 2\,\gc-2\\
\forall j\in J,\quad \psi_j\geq 0
} \nuXJ.(q^{(1-\gc)\dim(X)+\Theta}-\fonc^{\ast}).
\end{equation}
Si $k=\bQ$, on définit de manière analogue $n^J_{0,p,r}(y)$, $n^J_{1,p,r}(y)$ et $n^J_{2,p,r}(y)$.
\begin{thm}\label{thm:main}
\begin{enumerate}
\item
Soit $y\in \Pic(X)^{\vee}\cap \ceff(X)^{\vee}$. Si $k$ est fini, on a
\begin{equation}
\# \HOM_{X_0}(\courbe,X,y)(k)=n^J_0(y)+n^J_1(y)+n^J_2(y).
\end{equation}
Si $k=\bQ$, on a pour presque tout $p$ et tout $r\geq 1$
\begin{equation}
\# \HOM_{X_0}(\courbe,X,y)(\fpr)=n^J_{0,p,r}(y)+n^J_{1,p,r}(y)+n^J_{2,p,r}(y).
\end{equation}
\item
Pour $\alpha\in\{I\text{ fort},II\}$, 
$n^J_0$ est un terme dominant de type $\alpha$~;
$\{n^J_{0,p,r}\}$ est un terme dominant.
\item
Pour $\alpha\in\{I\text{ fort},II\}$, 
$n^J_0$ est un terme d'erreur de type $\alpha$~;
$\{n^J_{1,p,r}\}$ est un terme d'erreur.
\item
Soit 
$\cC_{\lambda}
\eqdef
\{y\in \ceff(X)^{\vee},\quad \acc{y}{(1-\lambda).(\scG_{\ju
  }+\scG_{\jd })-\Dtot}\geq 0
\}$.
Pour tout $\lambda>0$,
et pour $\alpha\in\{I\text{ fort},II\}$, 
$n^J_2$ est un terme d'erreur de type $\alpha$ sur $\cC_{\lambda}$~;
pour tout $\lambda>0$, $\{n^J_{2,p,r}\}$ est un terme d'erreur sur $\cC_{\lambda}$.
\end{enumerate}
\end{thm}
Avant d'en venir à la démonstration de ce théorème, nous en donnons
quelques corollaires immédiats.
\begin{cor}
Pour $\alpha\in \{I,I\text{ fort},II\}$, \og Manin $\alpha$\fg\ vaut
pour $(X,X_0)$ si et seulement si $n^J_2$ est 
un terme d'erreur de
type $\alpha$. \og Manin géométrique\fg\ 
vaut pour $(X,X_0)$ si et seulement si $\{n^J_{2,p,r}\}$ est un terme d'erreur.
\end{cor}
\begin{rem}
L'étude de $n^J_2$ dans le présent travail est basée sur la majoration
grossière \eqref{eq:maj:grossiere}. Il semble clair qu'en général le membre de
droite de \eqref{eq:maj:grossiere} ne sera pas un terme d'erreur,
et que seule une étude fine de la différence $q^{(1-\gc)\dim(X)+\Theta}-\fonc^{\ast}$ 
pourra permettre de montrer que $n^J_2$ est bien un terme d'erreur. 
\end{rem}
Pour $\lambda\geq 0$, soit 
\begin{equation}
\cP_{\lambda}=\cupu{J}\,\cC_{\lambda},\quad 
\cP'_{\lambda}=\ceff(X)^{\vee}
\setminus 
\capu{J}\,\{\acc{\bullet}{(1-\lambda).(\scG_{\ju}+\scG_{\jd})-\Dtot}\leq 0\}
\end{equation}
où $J\subset \indsec$ décrit à chaque fois l'ensemble des choix admissibles de
variables, de sorte que $\adh{\cP'_{\lambda}}=\cP_{\lambda}$.
\begin{cor}
Pour tout $\lambda>0$, \og Manin $I$ fort\fg\ vaut pour $(X,X_0)$ sur $\cP_{\lambda}$, 
et \og Manin $II$\fg\  et \og Manin géométrique\fg\ valent pour $(X,X_0)$ sur $\cP'_{\lambda}$.

S'il existe $\lambda>0$ tel que $\cP_{\lambda}=\ceff(X)^{\vee}$, \og
Manin $I$ fort\fg\ vaut pour $(X,X_0)$.

S'il existe $\lambda>0$ tel qu'on ait
$int(\ceff(X)^{\vee})\subset \cP'_{\lambda}$, \og Manin $II$\fg\  
et \og Manin géométrique\fg\ valent pour $(X,X_0)$.
\end{cor}
\begin{cor}\label{cor:manin:delpezzo}
\og Manin $I$ fort\fg, \og Manin $II$\fg,  et \og Manin géométrique\fg\ valent
pour les trois surfaces de del Pezzo généralisées suivantes~: la sextique
avec singularité $\bA_1$, la quintique avec singularité $\bA_2$ et la
quartique avec singularités $\bA_3+\bA_1$.
Pour les cinq surfaces de del Pezzo généralisées suivantes~: la sextique avec singularité $\bA_2$, les
quintiques avec singularités respectives $\bA_1$ et $\bA_2$, les quartiques
avec singularités respectives $3\,\bA_1$, $\bA_2+\bA_1$, et $\bA_3$, on a
\begin{equation}
\limsup_{\,\cP,\,\,\text{Manin vaut sur }\cP} \frac{\Vol(\cP)}{\Vol(\ceff(X)^{\vee})}=1.
\end{equation}
\end{cor}
Le dernier corollaire découle
aussitôt du corollaire précédent et d'un calcul Maple utilisant 
\cite{Fra:conv} et les données fournies par \cite{Der:sdp:ut:hyp}.
Pour toutes ces surfaces, on vérifie en effet qu'on a $\cP_0=\ceff(X)^{\vee}$ et
pour les trois premières on a $\cP'_{\lambda}\cup \{0\} =\ceff(X)^{\vee}$
pour $\lambda>0$ assez petit. 

On obtient d'ailleurs en toute généralités les résultats numérique présentés dans la table  
\ref{tab:results} pour les 20 surfaces de \cite{Der:sdp:ut:hyp} qui sont des hypersurfaces linéaires intrinsèques.
On notera en particulier que les résultats présentés dans la table 1
de \cite{Bou:moduli} sont sensiblement améliorés.
\begin{table}[ht]
  \centering
\[\begin{array}{|c|c|c|}
  \hline
\text{degré} 
& 
\text{singularités} 
& 
\limsup_{\,\cP,\,\,\text{Manin vaut sur }\cP} \frac{\Vol(\cP)}{\Vol(\ceff(X)^{\vee})}
\\
  \hline
  \hline
  6 & \bA_1 & 1\\
  6 & \bA_2 & 1\\
  5 & \bA_1 & 1\\
  5 & \bA_2 &1\\
  5 & \bA_3 &\geq 0.81\\
  5 & \bA_4 &\geq 0.76\\
  4 &  3\bA_1& 1\\
  4 & \bA_2+\bA_1 &1\\
  4 & \bA_3+\bA_1 &1\\
  4 & \bA_3 &  1\\
  4 & \bA_4 &  \geq 0.72\\
  4 & \bD_4 &  \geq 0.43\\
  4 & \bD_5 &  \geq 0.34\\
  3 & 2\bA_2+\bA_1 & \geq 0.97\\
  3 & \bA_3+2\,\bA_1 & \geq 0.92\\
  3 & \bA_4+\bA_1 & \geq 0.98\\
  3 & \bA_5+\bA_1 & \geq 0.68\\
  3 & \bD_4 & \geq 0.95\\
  3 & \bD_5 & \geq 0.28\\
  3 & \bE_6 & \geq 0.08\\
\hline
\end{array}\]
    \caption{}
  \label{tab:results}
\end{table}
\begin{rem}
En employant des techniques strictement similaires à celles utilisées
ici, on pourrait en fait montrer qu'on peut remplacer dans les
résultats ci-dessus l'ouvert $X_0$ par $X$. 
En particulier, on obtiendrait pour les degrés
$y$ ad hoc que $\HOM_{X_0}(\courbe,X,y)$ est dense dans $\HOM(\courbe,X,y)$.
Les détails seraient cependant assez laborieux (\cf \cite{Bou:moduli}
où ce type d'énoncé est considéré).
\end{rem}

Passons à la démonstration du théorème \ref{thm:main}.
Le premier point du théorème découle aussitôt de
\eqref{eq:form:relevement}, \eqref{eq:form:relevement:pr}
et des définitions.
La démonstration des autres points occupe la section suivante.

\section{Démonstration du théorème \ref{thm:main}}

\subsection{Démonstration du point 2 du théorème \ref{thm:main}}
Supposons $k$ fini.
D'après \eqref{eq:defn0}, \eqref{eq:def:scn},
\eqref{eq:theta} et \eqref{eq:adj}, on a
\begin{equation}
q^{(\gc-1)\dim(X)+\acc{y}{\can{X}}}
n^J_0(y)
=
\sumsub{
\becG\in \diveffc^J
\\
\deg(\ecG_j)\leq \acc{y}{\scG_j},\quad j\in J
}
q^{-\sum_{i\in I}\acc{y}{\scF_i}}
\scM((\acc{y}{\scF_i}),\becG)
q^{-\abs{\deg(\becG)}}
\end{equation}
et donc d'après le point
\ref{item:prop:res:sergen:kfini}
de la proposition \ref{prop:res:sergen}
on a pour tout $\eps>0$ assez petit
\begin{align}
\abs{
q^{\,\acc{y}{\can{X}}}
n^J_0(y)
-
\gamma(X)
}
\hskip-0.3\textwidth&\\
&
\ll_{\eps}
\sum_{i\in I}
q^{-\eps \acc{y}{\scF_i}}
\sum_{\becG\in \diveffc^J_{\leq 1}}
\left(
\prod_{v,v(\becG)\neq 0}C.q_v^{\abs{v(\becG)}-2+\eps}
\right)
q^{-\abs{\deg(\becG)}}
\notag
\\
&
\quad\quad+q^{-\eps \sum_{j\in J}\acc{y}{\scG_j}}
&
\\
&
\ll_{\eps}
q^{-\eps \sum_{j\in J}\acc{y}{\scG_j}}
+
\sum_{i\in I}
q^{-\eps \acc{y}{\scF_i}}.
\end{align}
Ceci montre le résultat pour $n^J_0$.

~

Supposons à présent $k=\bQ$. On a
\begin{equation}
p^{\,r\left[(\gc-1)\dim(X)+\acc{y}{\can{X}}\right]}
n^J_{0,p,r}(y)
=
\sumsub{\becG\in \diveffcpr^J
\\
\deg(\ecG_j)\leq\acc{y}{\scG_j},\quad j\in J
}
p^{\,-r\sum_{i\in I}\acc{y}{\scF_i}}
\scM_{p,r}((\acc{y}{\scF_i}),\becG)
p^{\,-r\abs{\deg(\becG)}}
\end{equation}
et donc d'après le point \ref{item:prop:res:sergen:kQ} 
de la proposition \ref{prop:res:sergen}, 
pour tout $\eps>0$ assez petit,
\begin{equation}
\abs{
p^{\,r\left[(\gc-1)\dim(X)+\acc{y}{\can{X}}\right]}
n^J_{0,p,r}(y)
-c(p,r)
}
\end{equation}
est majoré par 
\begin{equation}
C(p,r,\eps)
\Big[
\sum_{i\in I}p^{-r\eps\acc{y}{\scF_i}}
+
\sumsub{\becG\in \diveffcpr^J_{\leq 1}
\\
\becG\neq (0,\dots,0)
}
\left(
\prod_{v,v(\becG)\neq 0}
C.p_v^{\,r\,\abs{v(\becG)}-2+\eps}
\right)
p^{-r\abs{\deg(\becG)}}.
\Big]
\end{equation}
avec $\lim_{r\to +\infty}C(p,r,\eps)=1$.
D'après \cite[Lemma 5.2]{Bou:moduli}, la dernière somme tend vers $0$
quand $r\to +\infty$. On en déduit le résultat pour $n^J_{0,p,r}$.

\subsection{Démonstration du point 3 du théorème \ref{thm:main}}
Supposons $k$ fini.
On a 
\begin{equation}
\abs{n^J_1(y)}
\ll
\sum^y_{
\exists j\in J,\,\,
\psi_j<0} 
\abs{\nuXJ}.q^{\Theta}
+
\sum_{\jz\in J}
\,\,
\sumsuby{
\varphi_{\jz}\geq 2\,\gc-1\\
\forall j\in J,
\,\, 
\psi_j\geq 0} 
\abs{\nuXJ}.
\abs{q^{(\gc-1)\dim(X)+\Theta}-\fonc^{\ast}}
\end{equation}
et pour tout $\eta\leq 1$
\begin{equation}
\sum^y_{
\exists j\in J,
\,\,
\psi_j<0} 
\abs{\nuXJ}.q^{\,\Theta}
\ll
\sum_{\jz\in J}
\sum^y
\abs{\nuXJ}.
q^{\,\varphi_{\jz}+(1-\eta)\psi_{\jz}}.
\end{equation}
Par ailleurs, d'après le point 
\ref{item:lm:compt:1} 
du lemme \ref{lm:compt}, on a
\begin{multline}
\sumsuby{
\varphi_{\jz}\geq 2\,\gc-1\\
\forall j\in J,\quad \psi_j\geq 0
} 
\abs{\nuXJ}.
\abs{q^{(\gc-1)\dim(X)+\Theta}-\fonc^{\ast}}
\\
=
\sumsuby{
\varphi_{\jz}\geq 2\,\gc-1\\
\forall j\in J,\quad \psi_j\geq 0
\\
\psi_{\jz}\geq 2\,\gc-1}  
\abs{\nuXJ}.\abs{-1+\sum_{j\in J} \fonc_j}
+
\sumsuby{
\varphi_{\jz} \geq 2\,\gc-1\\
0\leq \psi_{\jz}\leq 2\,\gc-2} 
\abs{\nuXJ}
\abs{q^{\,(1-\gc)\dim(X)+\Theta}-\fonc^{\ast}}.
\end{multline}
En utilisant le point \ref{item:lm:compt:2}
du lemme \ref{lm:compt}
on obtient pour tout $\eta\leq 1$ 
\begin{align}
\sumsuby{
\varphi_{\jz}\geq 2\,\gc-1\\
\forall j\in J,\quad \psi_j\geq 0
\\
\psi_{\jz\geq 0}
}  
\abs{\nuXJ}.\abs{-1+\sum_{j\in J} \fonc_j}
&
\ll
\sum_{j\in J}
\,\,
\sumsuby{\forall j\in J,\,\,\psi_j\geq 0}
\abs{\nuXJ}
q^{\,\varphi_{\jz}}
\\
&
\ll
\sum_{j\in J}
\,\,
\sum^y
\abs{\nuXJ}
q^{\varphi_{\jz}+(1-\eta)\psi_{\jz}}.
\end{align}
Finalement, en utilisant le point \ref{item:lm:compt:dern} du lemme 
\ref{lm:compt} on obtient pour $\eta\geq 0$
\begin{equation}
\sumsuby{
\varphi_{\jz} \geq 2\,\gc-1\\
0\leq \psi_{\jz}\leq 2\,\gc-2} 
\abs{\nuXJ}
\abs{q^{(1-\gc)\dim(X)+\Theta}-\fonc^{\ast}}
\ll_{\eta}
\sum^y 
\abs{\nuXJ}q^{\varphi_{\jz}+(1-\eta)\psi_{\jz}}.
\end{equation}
Pour $0\leq \eta\leq 1$ on obtient en définitive
\begin{equation}
\abs{n^J_1(y)}\ll_{\eta}\sum_{j\in J} \sum^y \abs{\nuXJ}.q^{\,\varphi_j+(1-\eta)\psi_j}.
\end{equation}
Notons que d'après \eqref{eq:defjzeta}, \eqref{eq:psijz},
\eqref{eq:phijz}
et \eqref{eq:adj}
on a pour tout $\eta$ et tout $j\in J$
\begin{multline}
q^{\acc{y}{\can{X}}}
\sum^y \abs{\nuXJ}.q^{\,\varphi_j+(1-\eta)\psi_j}
\\
=
q^{-\eta\acc{y}{\scG_j}} 
\sum_{\becG\in \diveffc^J_{\leq 1}}
q^{-\sum_{i\in I}\acc{y}{\scF_i}}
\scM_{_{j,\eta}}(\acc{y}{\scF_i},\becG)
q^{-\abs{\deg(\becG)}+\eta\deg(\ecG_j)}.
\end{multline}
En appliquant le point \ref{item:prop:res:sergen:kfini} de la
proposition \ref{prop:res:sergen},
pour tout $\eps>0$ assez petit, on obtient pour tout $\eta>0$
assez petit la majoration
\begin{equation}
q^{\acc{y}{\can{X}}}\abs{n^J_1(y)} 
\ll_{\eps,\eta} 
\sum_{j\in J}q^{-\eta\acc{y}{\scG_j}}
\end{equation}
d'où le résultat.

Supposons $k=\bQ$.
Par un raisonnement analogue à ci-dessus, on a
pour tout $\eta$
\begin{equation}
\abs{n^J_{1,p,r}(y)}\ll_{\eta}p^{r\,(\dim(X)+\eta (2\,\gc-2))}
\sum_{j\in J} \sum^y \abs{\nuXJpr}.p^{r(\varphi_j+(1-\eta)\psi_j)}
\end{equation}
Pour tout $j\in J$ et tout $\eps>0$ assez petit, en appliquant le point 
\ref{item:prop:res:sergen:kQ}
de la proposition 
\ref{prop:res:sergen},
on obtient pour $\eta>0$ assez petit 
\begin{align}
p^{\,r((\gc-1)\dim(X)+\acc{y}{\can{X}})}
p^{r\,(\dim(X)+\eta (2\,\gc-2))}
\sum^y \abs{\nuXJpr}.p^{\,r(\varphi_j+(1-\eta)\psi_j)}
\hskip-0.7\textwidth
&
\\
&
=
p^{r(\gc\dim(X)+\eta(2\,\gc-2)-\eta\acc{y}{\scG_j})} p^{-r\sum_{i\in I}\acc{y}{\scF_i}}
\notag
\\
&
\quad 
\times
\sum_{\becG\in \diveffcpr^J_{\leq 1}}
\scM_{j,\eta,p,r}(\acc{y}{\scF_i},\becG)
p^{-r(\abs{\deg(\becG)}+\eta\deg(\ecG_j))}
\\
&
\leq
C(p,r,\eps,\eta)\,p^{r(\gc\dim(X)+\eta(2\,\gc-2)-\eta\acc{y}{\scG_j})}
\end{align}
avec $\lim_{r\to +\infty}C(p,r,\eps,\eta)=1$ (on a encore utilisé \cite[Lemma 5.2]{Bou:moduli})
d'où le résultat (rappelons que l'on peut exprimer explicitement un
$\eta$ convenable en fonction de $X$).

\subsection{Démonstration du point 4 du théorème \ref{thm:main}}
Supposons $k$ fini. 
\'Ecrivons   $J=\{\jz,\ju,\jd\}$.
D'après les points \ref{item:lm:compt:avant:dern}
et \ref{item:lm:compt:dern}
du lemme \ref{lm:compt}, on a la majoration
\begin{equation}\label{eq:maj:grossiere}
\abs{n^J_2(y)}
\ll
\sum^y\abs{\nuXJ}.q^{\psi_{\jz}}.
\end{equation}
Soit $y$ vérifiant $\acc{y}{\scG_j}\leq \deg(\ecG_j)$ pour tout $j\in J$.
D'après \eqref{eq:psijz}, on a pour $\lambda>0$ 
\begin{multline}\label{eq:maj:psi1:bis}
\psi_{\jz}(y,\becG,\becD)
\leq
\acc{y}{\sum_{j\in J} \scG_j-\Dtot}
-\acc{y}{(1-\lambda)\,\scG_{\ju}+\scG_{\jd}-\Dtot}
-\deg(\ecG_1)
\\
+\deg \pgcdu{j\in J}\left(\ecG_j+\sum_{i\in I_j} b_{i,j}\ecD_i\right)
-\lambda(\deg(\ecG_{\ju})+\deg(\ecG_{\jd}))
-\deg(\pgcd(\ecG_{\ju},\ecG_{\jd})).
\end{multline}
En utilisant \eqref{eq:maj:ceps}, on obtient pour tout $\eps>0$ assez
petit la majoration
\begin{multline}\label{eq:expr:psi:bis}
q^{\,\acc{y}{\can{X}}}\,
\sum^y\abs{\nuXJ}.q^{\psi_{\jz}}
\\
\ll_{\eps}
\sum_{
\substack{
\becG\in \diveffc^J_{\leq 1}
}
}
\left(
\prod_{v,v(\becG)\neq 0}
C.
q_v^{\abs{v(\becG)}-2+\eps}
\right)
q^{-\deg(\ecG_{\jz})-\deg(\pgcd(\ecG_{\ju},\ecG_{\jd}))-\lambda(
    \deg(\ecG_{\ju})-\deg(\ecG_{\jd}))}\,
\\
\times 
q^{-\acc{y}{(1-\lambda)\scG_{\ju}+\scG_{\jd}-\Dtot}}
\end{multline}
On vérifie aussitôt qu'on a pour $\bg\in \{0,1\}^J$
et $\eps>0$ assez petit
\begin{equation}\label{eq:la:maj:bis}
\abs{\bg}-2+\eps-g_{\jz}-\Min(g_{\ju},g_{\jd})-\lambda(g_{\ju}+g_{\jd})<-1
\end{equation}
d'où la convergence de la série apparaissant dans \eqref{eq:expr:psi:bis}
et la majoration 
\begin{equation}
q^{\,\acc{y}{\can{X}}}\abs{n^J_2(y)}\leq_{\eps,\lambda}
q^{-\acc{y}{(1-\lambda)\scG_2+\scG_3-\Dtot}}.
\end{equation}
Si $k=\bQ$, on a de même la majoration
\begin{equation}\label{eq:maj:grossiere:pr}
p^{\,r[(1-\gc)\dim(X)]}
\abs{n^J_{2,p,r}(y)}
\ll
p^{\,r\,(\gc\dim(X)+2\,\gc-2)}
\sum^y\abs{\nuXJpr}.p^{r\,\psi_{\jz}}.
\end{equation}
et un raisonnement analogue permet de conclure.

\section{Appendice~: démonstration du quatrième point de la
  proposition \ref{prop:res:sergen}
}
On commence par remarquer que
$
F_{\jz,2,\bg}(\rho,\tau,\bt)
$
est la somme des trois séries suivantes~:
\begin{equation}\label{eq:series:1}
\sumsub{
\bbf\in \bN^I
\\
(0,\bbf)\in \Nindprim
\\
\Minu{j\in J\setminus \{\jz\}}(g_j+\sum_{i\in I_j} a_{i}f_i)\geq 2
\\
g_{\jz}+\sum_{i\in I_{\jz}} a_{i}f_i=0
}
\abs{\nuzXJ (\bg,\bbf)}
(
\tau^{\,\Minu{j\in J\setminus \{\jz\}}(g_j+\sum_{i\in I_j} a_{i}f_i)}
-\tau
)
\bt^{\bbf},
\end{equation}
\begin{equation}\label{eq:series:2}
\sumsub{
\bbf\in \bN^I
\\
(0,\bbf)\in \Nindprim
\\
\Minu{j\in J\setminus \{\jz\}}(g_j+\sum_{i\in I_j} a_{i}f_i)\geq 2
\\
g_{\jz}+\sum_{i\in I_{\jz}} a_{i}f_i=1
}
\abs{\nuzXJ (\bg,\bbf)}
(
\rho\,\tau^{\,\Minu{j\in J\setminus \{\jz\}}(g_j+\sum_{i\in I_j} a_{i}f_i)}
-\rho\,\tau
)
\bt^{\bbf},
\end{equation}
et
\begin{equation}\label{eq:series:3}
\sumsub{
\bbf\in \bN^I
\\
(0,\bbf)\in \Nindprim
\\
\Minu{j\in J}(g_j+\sum_{i\in I_j} a_{i}f_i)\geq 2
}
\abs{\nuzXJ (\bg,\bbf)}
(
\rho^{\,\Minu{j\in J}(g_j+\sum_{i\in I_j} a_{i}f_i)}\,
\tau^{\,\Minu{j\in  J\setminus \{\jz\}}(g_j+\sum_{i\in I_j} a_{i}f_i)}
-\rho\,\tau
)
\bt^{\bbf}.
\end{equation}
Il s'agit de montrer que ces trois séries sont
$\abs{\bg}$-contrôlées. 
\paragraph{Le cas de \eqref{eq:series:1}}
On peut bien entendu supposer qu'on a $g_{\jz}=0$
Soit $\bbf\in \bN^I$ vérifiant les conditions sous le signe somme dans \eqref{eq:series:1}.
D'après l'hypothèse \eqref{hyp:noninter},
on est alors nécessairement dans l'une des deux situations suivantes~:
\begin{itemize}
\item
il existe $K\in \classadminc$ de cardinal $\#J-1$
tel que $K\cap I_j$ est
un singleton pour tout $j\in J\setminus \{\jz\}$ 
et $f_i\geq 1$ si et seulement si $i\in K$
ou $i\notin \disju{j\in J}I_j$ et  
$\{i\}\cup K\in \classadminc$~;
\item
il existe $K\in \classadminc$ de cardinal $\#J$ et $\ju \in J\setminus
\{\jz\}$
tel que $\# (K\cap I_{\ju })=2$ (posons alors $\{\iu,\ideux\}\eqdef K\cap I_{\ju }$),
$K\cap I_j$ est un singleton pour $j\notin \{\ju ,\jz\}$ 
et $f_i=1$ si et seulement si $i\in K$. 
\end{itemize}
\eqref{eq:series:1} se réécrit ainsi comme une somme sur tous les
ensembles $K$ du premier type des séries
\begin{equation}\label{eq:series:1:A}
\sumsub{
i\notin \disju{j\in J}I_j, 
\\
K\cup \{i\}\in \classadminc
}
\abs{\nuzXJ (\bg,\bbf_{K\cup \{i\}})}
\frac{t_i}{1-t_i}
\sumsub{
\bbf\in \bN^{J\setminus \{\jz\}}_{>0}
\\
\forall j\in J_1,\,\,g_j=0\imply f_j\geq 2
}
(
\tau^{\,\Minu{j\in J\setminus \{\jz\}}(g_{j}+a_j f_j)}
-
\tau
)
\bt^{\bbf}
\end{equation}
et
\begin{equation}\label{eq:series:1:B}
\abs{\nuzXJ (\bg,\bbf_{K})}
\sumsub{
\bbf\in \bN^{J\setminus \{\jz\}}_{>0}
\\
\forall j\in J_1,\,\,g_j=0\imply f_j\geq 2
}
(
\tau^{\,\Minu{j\in J\setminus \{\jz\}}(g_{j}+a_j f_j)}
-
\tau
)
\bt^{\bbf}
\end{equation}
(on a identifié $K$ à $J\setminus \{\jz\}$ et noté
$J_1=\{j\in J\setminus \{\jz\},\,a_j=1\}$),
plus une somme sur tous les ensembles $K$ du second type des séries
\begin{equation}\label{eq:series:1:C}
\abs{\nuzXJ (\bg,\bbf_K)}
\sumsub{
\bbf\in \bN^{
J\setminus \{\jz,\ju\}
}_{>0}
\\
\forall j\in J_1,\,\,g_j=0\imply f_j\geq 2
}
(
\tau^{\,\Min(g_{\ju }+a_{\iu }f_{\iu }+a_{\ideux }f_{\ideux },(g_{j}+a_j f_j)_{{j\in J\setminus \{\jz,\ju \}}})}
-
\tau
)
\bt^{\bbf}
\end{equation}
(on a identifié $K\setminus \{\iu ,\ideux \}$ à 
$J\setminus \{\jz,\ju\}$ et noté
$J_1=\{j\in J\setminus \{\jz,\ju\},\,a_j=1\}$).

D'après la proposition \ref{prop:calc:sergen},
la somme sur 
$\bbf\in \bN^{J\setminus \{\jz\}}_{>0}$ dans 
\eqref{eq:series:1:A} et \eqref{eq:series:1:B} vaut, en 
notant $m$ le plus petit commun multiple des $(a_j)_{j\in J\setminus  \{\jz\}}$, 
\begin{equation}\label{eq:premier:type}
\frac{
\tau
\produ{j\in J_1}t_j^{2-g_j}
\produ{j\in 
J\setminus  (\{\jz\}\cup J_1)}
t_j
}
{
\produ{j\in J\setminus \{\jz\}} (1-t_j)
}
\left[
\frac{
\tau\,\wt{\cF}(\tau,(a_j)_{j\in J\setminus
    \{\jz\}},\bnu,\bt)}
{
(1-\tau^{\,m}\,\produ{j\in J\setminus \{\jz\}}  t_j^{\frac{m}{a_j}})
}
-
1
\right]
\end{equation}
avec $\bnu=((0,\dots,0)_{j\in J_1},(g_j+a_j-2)_{j\in J\setminus (\{\jz\}\cup J_1)})$.

Toujours d'après la proposition \ref{prop:calc:sergen},
la somme sur
$\bbf\in \bN^{
J\setminus \{\jz,\ju\}
}_{>0}
$
dans 
\eqref{eq:series:1:C}
vaut,
en notant $m_1$ le plus petit commun multiple des 
$\{a_{\,\iu },(a_j)_{j\in J\setminus  \{\jz,\ju\}}\}$, 
et $m_2$ le plus petit commun multiple des 
$\{a_{\ideux },(a_j)_{j\in J\setminus  \{\jz,\ju\}}\}$, 
\begin{equation}\label{eq:second:type}
\frac{
\tau
\,
t_{\iu }
\,
t_{\ideux }
\produ{j\in J_1}t_j^{2-g_j}
\produ{j\in J\setminus (J_1\cup \{\jz,\ju)\}}t_j
}
{
(1-t_{\iu })(1-t_{\ideux })
\produ{j\in J\setminus \{\jz,\ju\}} (1-t_j)
}
\left[
\frac{\rho\,\wt{\cF}(\tau,(a_{\iu },a_{\ideux },\{a_j\}_{j\in J\setminus \{\jz,\ju\}}),\bnu,\bt)}{
(1-\tau^{\,m_1}\,t_{\iu }^{\frac{m_1}{a_{\iu }}}
\produ{j\in J\setminus \{\jz,\ju\}}  t_j^{\frac{m_1}{a_j}})
(1-\tau^{\,m_2}\,
t_{\ideux }^{\frac{m_2}{a_{\ideux }}}
\produ{j\in J\setminus \{\jz,\ju\}}  t_j^{\frac{m_2}{a_j}})}
-1
\right]
\end{equation}
avec 
$\bnu=((0,\dots,0)_{j\in J_1},(g_j+a_j-2)_{j\in J\setminus (J_1\cup \{\jz,\ju\})},(g_{\ju}+a_{\iu }+a_{\ideux }-2))$.
Compte tenu de $\# J\geq 3$ et de la proposition \ref{prop:calc:sergen}, 
on vérifie facilement que \eqref{eq:premier:type} et \eqref{eq:second:type} sont $\abs{\bg}$-contrôlée.
\paragraph{Le cas de \eqref{eq:series:3}}
En raisonnant comme pour le deuxième point de la proposition 
\ref{prop:res:sergen},
on voit qu'il suffit de montrer que pour tout $K\in \scI$  
la série 
\begin{equation}\label{eq:serie}
\abs{\nuzXJ (\bg,\bbf_K)}
\sumsub{
\bbf\in \bN^J_{>0}
\\
\forall j\in J_1,\,\,g_j=0\imply f_j\geq 2
}
(
\rho^{\,\Minu{j\in J}(g_j+a_{j}f_j)}
\tau^{
\eps\,\Minu{j\in
    J\setminus \{\jz\}}(g_j+ a_{j}f_j)
}
-\rho\tau
)
\bt^{\bbf}
\end{equation}
(on a identifié $K$ à $J$ et posé $J_1\eqdef\{j\in J, a_j=1\}$)
est $\abs{\bg}$-contrôlée. Rappelons que par hypothèse $J_1$ est non vide.
Notons $m$ le plus petit commun
multiple de la famille $(a_j)_{j\in J}$ et 
$n$ le plus petit commun
multiple de la famille $(a_j)_{j\in J\setminus \{\jz\}}$.
D'après la proposition \ref{prop:calc:sergen}, \eqref{eq:serie} s'écrit,
modulo le facteur $\nuzXJ $,
\begin{equation}\label{eq:series:bis}
\frac{
\rho\tau
\prod_{j\in J_1} t_j^{2-g_j}
\prod_{j\notin J\setminus J_1} t_j
}
{
\produ{j\in J} (1-t_j)
}
\left[
\frac{\tau\,\,\wt{\cG}(\jz,\rho,\tau,(a_j)_{j\in J},\bnu,\bt)}{
(1-\rho^{\,m}\tau^{\, n}\,\produ{j\in J}  t_j^{\frac{m}{a_j}})
(1-\tau^{\,n}\,\produ{j\in J\setminus \{\jz\}}  
t_j^{\frac{n}{a_j}})
}
-
1
\right]
\end{equation}
avec $\bnu=((0,\dots,0)_{j\in J_1},(g_j+a_j-2)_{j\notin J_1})$.
D'après la proposition \ref{prop:calc:sergen},
on voit facilement que \eqref{eq:series:bis} est $\abs{\bg}$-contrôlée.
\paragraph{Le cas de \eqref{eq:series:2}}
Dans le cas où $g_{\jz}=1$, on se ramène facilement au cas de
\eqref{eq:series:1}.
Dans le cas où $g_{\jz}=0$ 
on voit qu'il suffit de montrer que pour tout $K\in \scI$  
contenant $\jz$ la série
\begin{equation}\label{eq:serie:bis}
\abs{\nuzXJ (\bg,\bbf_K)}
t_{\jz}
\sumsub{
\bbf\in \bN^{J\setminus \{\jz\}}_{>0}
\\
\forall j\in J_1,\,\,g_j=0\imply f_j\geq 2
}
(
\rho
\tau^{\,
\Minu{j\in
    J\setminus \{\jz\}}(g_j+ a_{j}f_j)
}
-\rho\tau
)
\bt
\end{equation}
(on a identifié $K$ à $J$ et posé $J_1\eqdef\{j\in J\setminus \{\jz\}, a_j=1\}$)
est $\abs{\bg}$-contrôlée. Comme précédemment, ceci découle aisément d'une
application de la proposition \ref{prop:calc:sergen}.

\bibliographystyle{alpha}

\end{document}